\documentclass[11pt]{amsart}
\usepackage{categorytheory}
\usepackage{stmaryrd}
\fullpage
\usetikzlibrary{calc,decorations.pathmorphing}

\newcommand{\SCob}[2]{\{#1_{#2}\}_{#2\in\uppercase{#2}}}

\newcommand{\F}{\mathscr{F}}
\newcommand{\Assemb}{\mathbf{Asm}}
\newcommand{\bs}{\smallsetminus}
\newcommand{\Spc}{\mathbf{Spc}}
\newcommand{\WaldCat}{\mathbf{WaldCat}}
\newcommand{\W}{\mathcal{W}}
\DeclareMathOperator{\SC}{SC}
\DeclareMathOperator{\Sym}{Sym}
\DeclareMathOperator{\Tw}{Tw}
\DeclareMathOperator{\Spec}{Spec}
\DeclareMathOperator{\Frob}{Frob}
\newcommand{\Gal}{\mathrm{Gal}}
\newcommand{\Fin}{\mathbf{Fin}}

\newcommand{\AFSet}{\mathbf{AFSet}}

\newcommand{\FF}{\mathbb{F}}
\newcommand{\ab}{\mathrm{ab}}
\renewcommand{\emph}{\textsl}
\newcommand{\RR}{\mathbf{R}}
\newcommand{\Seg}{\mathbf{Seg}}
\newcommand{\NN}{\mathbf{N}}

\newcommand{\PP}{\mathbb{P}}
\newcommand{\ord}{\mathrm{ord}}

\def\morpair#1#2#3#4{\begin{tikzpicture}[baseline=(m-1-1.base)] \node[anchor=east]%
    (m-1-1) at (0,0) {$#1$}; \node[anchor=west] (m-2-2) at (2em,0) {$#2$};%
    \diagArrow{->, bend left}{1-1}{2-2}!{#3} \diagArrow{->, bend %
      right}{1-1}{2-2}!{#4}%
  \end{tikzpicture}}
\def\longmorpair#1#2#3#4#5{\begin{tikzpicture}[baseline=(m-1-1.base)] \node[anchor=east]%
    (m-1-1) at (0,0) {$#2$}; \node[anchor=west] (m-2-2) at (#1,0) {$#3$};%
    \diagArrow{->, bend left}{1-1}{2-2}!{#4} \diagArrow{->, bend %
      right}{1-1}{2-2}!{#5}%
  \end{tikzpicture}}

\begin{document}

\title{$E_\infty$-Ring structures on the $K$-theory of assemblers and point
  counting}
\author{Inna Zakharevich}
\maketitle

\begin{abstract}
  We construct a monoidal structure on the category of assemblers.  As an
  appication of this, we construct a derived local zeta-function which takes a
  variety over a finite field to the set of points over the separable closure,
  and use the structure of this map to detect interesting elements in
  $K_1(\Var_k)$. 
\end{abstract}

\section*{Introduction}

An additive invariant on varieties over a base field $k$ with values in an
abelian group $A$ is a function $\mu:\{\mathrm{Varieties}/k\} \rto A$ such that
for any closed immersion $Y \rcofib X$, $[X] = [Y] + [X \bs Y]$.  Such an
invariant must factor uniquely through a homomorphism from the Grothendieck ring
of varieties, $K_0(\Var_k)$, defined by
\[K_0(\Var_k) \defeq
  \begin{array}{c}
    \hbox{free ab. gp. gen.} \\ \hbox{by varieties over $k$}
  \end{array} \left/
    \begin{array}{ll}
      {} {Y \rcofib^{\mathrm{closed}} X} \\
      {} [X] = [Y] + [X \bs Y].
    \end{array}\right.\]
In other words, $K_0(\Var_k)$ is the \emph{universal} additive invariant: if all
structure on $K_0(\Var_k)$ could be understood then all additive invariants
would also be understood.  Such an invariant is called \emph{multiplicative} if
it takes values in a ring and satisfies the additional condition that
$\mu(X\times Y) = \mu(X)\mu(Y)$.  Defining the ring structure on $K_0(\Var_k)$
to be $[X][Y] = [X\times Y]$ implies that any multiplicative additive invariant
must factor uniquely through a ring homomorphism from $K_0(\Var_k)$.

The group $K_0(\Var_k)$ can be modeled topologically as the connected components
of a $K$-theory spectrum $K(\Var_k)$, introduced in \cite{Z-Kth-ass}.  The
higher homotopy groups of this spectrum encode further geometric information
about piecewise-isomorphisms of varieties.  The group $K_1(\Var_k)$ can be
thought of (by analogy with the $K$-theory of a ring) as a ``determinant'' for
piecewise-automorphisms of varieties, in the following manner.

Consider the determinant of a matrix over a field $F$.  The determinant is a
collection of homomorphisms $\det_n: GL_n(F) \rto F^\times$ for each positive
integer $n$ satisfying the following conditions:
\begin{description}
\item[additivity] for positive integers $m$ and $n$, the following diagram
  commutes:
  \begin{diagram}
    { GL_m(F) \oplus GL_n(F) & GL_{m+n}(F) \\
      F^\times \oplus F^\times & F^\times \\};
    \cofib{1-1}{1-2} \to{1-1}{2-1}_{\det_m\oplus \det_n}
    \to{1-2}{2-2}^{\det_{m+n}}
    \to{2-1}{2-2}^\cdot
  \end{diagram}
\item[initial conditions] the following computations hold:
  \[\textstyle{\det_2} \left(\begin{array}{cc} 0 & 1 \\ 1 & 0 \end{array}\right) = -1 \qqand
    \det|_{GL_1(F)} = \mathrm{id}.\]
\end{description}
In particular, the additivity of the determinant ensures \emph{stability}: the
homomorphisms $\det_n$ extend to a homomorphism $GL(F) \rto F^\times$.  If we
generalize $F$ to a ring we have a choice about whether to enforce the initial
conditions or not.  If we enforce the initial conditions we can use the standard
formula for the determinant and the additivity condition will hold; this
determinant will take values in $R^\times$.  However, it also makes sense to say
that the condition we truly care about is additivity; in this case, we can say
that the determinant should take values in the largest abelian group possible
(ensuring that additivity still holds)---the group $GL(R)^{ab} = K_1(R)$.  The
fact that $K_1(R)$ is not necessarily isomorphic to $R^\times$, as it is for a
field, demonstrates that for some base rings the determinant contains more
information than a single unit.  Moreover, although the determinant on $GL_1(R)$
is no longer the identity, there is a natural inclusion $GL_1(R) \rto K_1(R)$.
For a more in-depth discussion of $K_1(R)$, see \cite[Chapter III]{kbook}.

In the case of varieties there is a similar description of the determinant for
automorphisms of varieties.\footnote{This definition an discussion generalizes
  directly to piecewise-automorphisms of varieties; however, in the interest of
  readability we focus entirely on honest automorphisms here.}  For a variety
$X$, write $\Aut(X)$ for the group of automorphisms of $X$.  The additivity
condition for the determinant can be described as the fact that for every
variety $X$ there is a homomorphism
\[\textstyle{\det_X}:\Aut(X) \rto K_1(\Var_k)\]
satisfying the additivity condition that the diagram
\begin{diagram}
  { \Aut(X) \oplus \Aut(Y) & \Aut(X\amalg Y) \\
    K_1(\Var_k) \oplus K_1(\Var_1) & K_1(\Var_k) \\};
  \cofib{1-1}{1-2} \to{1-1}{2-1}_{\det_X\oplus \det_Y}
  \to{1-2}{2-2}^{\det_{X\amalg Y}} \to{2-1}{2-2}
\end{diagram}
commutes for any varieties $X$ and $Y$ over $k$.  The initial conditions are
more complicated, although the condition on the swap has a natural
interpretation.  Finite sets can be considered $0$-dimensional varieties, and
this induces a map $\S \simeq K(\Fin) \rto K(\Var_k)$.  The non-identity element
in the group $K_1(\Fin) \cong \pi_1\S \cong \Z/2$ thus has an image in
$K_1(\Var_k)$; this is the element which is the determinant of the ``swap''.
Note that this agrees with the classical definition: if finite sets are mapped
to vector spaces of the appropriate dimension, then the two-point swap is mapped
to the matrix $
\begin{matrix}{cc}
  0 & 1 \\ 1 & 0
\end{matrix}$ and the induced map on $K_1$ is exactly the homomorphism
$\Z/2 \rto F^\times$ taking $-1$ to itself.

\begin{remark*}
  It is natural to ask what the natural analog of the computation of the
  determinant of $GL_1(F)$ could be.  One possible candidate is the following:
  the group $\Aut(\PP^1)$ contains as a subgroup $k^\times$,
  where $a\in k^\times$ represents the automorphism $[x:y] \rgoesto [ax:y]$.
  There is therefore an induced map $k^\times \rto K_1(\Var_k)$. The question of
  whether this map is injective is still open.
\end{remark*}

In order to analyze elements in $K_1(\Var_k)$ it is possible to ``lift''
invariants of $K_0(\Var_k)$.  More rigorously, suppose that we are given a
homomorphism $K_0(\Var_k) \rto K_0(\C)$ for some category $\C$ which has an
associated $K$-theory spectrum (such as finite sets, projective $R$-modules,
etc.).  It is often possible to lift this homomorphism to a map of spectra
$K(\Var_k) \rto K(\C)$; this map will induce a map $K_n(\Var_k) \rto K_n(\C)$
for all integers $n$.  For example, when $k$ is finite, consider the usual local
zeta function of a variety $X$, written in terms of the action of Frobenius on
$\ell$-adic cohomology,
\[Z(X,t) = \prod_{i=0}^{2\dim X} \det (1- t \Frob|_{H^i_c(\bar X,
    \Q_\ell)})^{(-1)^{i+1}}.\]
The local zeta function induces a homomorphism $K_0(\Var_k) \rto
K_0(\mathrm{GalRep}(\Q_\ell))$, where $\mathrm{GalRep}(Q_\ell)$ is the category
of finitely-generated continuous Galois representations.  In \cite{CWZ-zeta} the
authors lift this homomorphism to a map of spectra and use it to show that when
$|k| \equiv 3 \pmod 4$ the group $K_1(\Var_k)$ contains elements which are not
in the image of $K_1(\Fin)$; in particular, they show that the element
$[\PP^1,1/x] \in \Aut(\PP^1)$ is not in the image of $K_1(\Fin)$.

However, the authors were not able to show that the induced map $K(\Var_k) \rto
K(\mathrm{GalRep}(\Q_\ell))$ is a map of $E_\infty$-ring spectra; in
less-technical language, the authors could lift the group homomorphism but not
the multiplicative structure to the map of spectra.  Although the spectrum
$K(\Var_k)$ was shown in \cite{campbell14} to have an $E_\infty$-structure, the
method for constructing the map of spectra could not ensure that the map is
compatible with this structure.

In this paper, we develop alternate machinery for constructing an
$E_\infty$-structure on $K(\Var_k)$ and use it to construct an
$E_\infty$-version of the local zeta function.  In order to ensure that the map
is $E_\infty$ we use a different, more combinatorial, model of the local zeta
function:
\[Z(X,t) = \exp \sum_{n \geq 1} \frac{|X(\FF_{q^n})|}{n} t^n,\]
when $k = \FF_q$.  Noting that the set $X(\FF_{q^n})$ is uniquely determined by
the set $X(\bar\FF_q)$ together with the action of Frobenius, we consider the
local zeta function to be a map $K_0(\Var_k) \rto K_0(\AFSet_{\hat Z})$.  Here
$\AFSet_{\hat Z}$ is the category of \emph{almost-finite sets}; see
Section~\ref{sec:afsets} for a rigorous definition.  This gives rise to the
following theorem:

\begin{maintheorem}[Theorems~\ref{thm:Einfty} and
  \ref{thm:ringhom}] \label{main:Einfty} The spectra $K(\Var_k)$ and
  $K(\AFSet_{\hat Z})$ are $E_\infty$-ring spectra.  The local zeta function
  induces a ring homomorphism $K_*(\Var_k) \rto K_*(\AFSet_{\hat Z})$.
\end{maintheorem}

This ring structure allows us to do a more in-depth analysis of the structure of
$K_1(\Var_k)$.  The $E_\infty$-structure on $K(\Var_k)$ induces a multiplication
\[K_0(\Var_k) \otimes K_1(\Fin) \rto K_0(\Var_k) \otimes K_1(\Var_k) \rto
  K_1(\Var_k).\]
The elements in this image of this map are those which can be represented by an
automorphism which takes two copies of some variety $X$ and swaps them.  The
element in the image of $K_1(\Fin)$ is one such element, but there are others as
well.  We call elements in this image \emph{permutative}, and those not in the
image \emph{non-permutative}.  The model of the local zeta function constructed
in this paper allows us to conclude the following:

\begin{maintheorem}[Corollary~\ref{cor:gen-root}]
  There exist non-permutative elements in $K_1(\Var_k)$ for all finite fields
  $k$ with odd characteristic.  In particular, let
  $b = \ord_2(|k|-1)$, and let $\alpha\in k$ be a primitive
  $2^b$-th root of unity.  Then the element $[\mathbb{P}^1, x \mapsto \alpha x]$
  is a non-permutative element of $K_1(\Var_k)$.  \note{What about even?}
\end{maintheorem}

\begin{remark*} Permutative elements exist in $K_n(\Var_k)$ for all positive $n$,
  not just $n=1$; the proof of \cite[Theorem 6.6]{CWZ-zeta} uses such elements
  to show that when $k$ is a subfield of $\CC$ there are infinitely many
  nontrivial groups $K_n(\Var_k)$.  (For finite fields this is clear, as the map
  $X \rgoesto X(k)$ induces a map $K(\Var_k) \rto K(\Fin)$ which splits the map
  $K(\Fin)\rto K(\Var_k)$.) 
\end{remark*}

The proofs of these theorems use the machinery of \emph{assemblers}, originally
introduced in \cite{Z-Kth-ass}.  Assemblers contain the combinatorial data of
how different objects of interest decomppose but strip out the specific data of
the context.  Their combinatorial nature makes them amenable to analytic
techniques, equipping them with d\'evissage and localization theorems analogous
to Quillen's theorems for abelian categories.

The assembler of varieties over $k$, generally denoted $\Var_k$ to be consistent
with the notation for the Grothendieck ring, has as objects the varieties over
$k$ and as morphisms locally-closed immersions.  It is also equipped with a
pretopology generated by the coverage $\{Y \rcofib X, X\bs Y \rcofib X\}$ (where
$Y \rcofib X$ is a closed immersion).  In this paper we show that the assembler
definition of the $K$-theory of varieties gives rise to an $E_\infty$-structure
on the $K$-theory by showing that the $K$-theory functor on assemblers is
``almost monoidal,'' taking monoid objects in assemblers to $E_\infty$-ring
spectra.  In fact, this can be generalized a bit further, showing that not only
monoid objects in the category of assemblers but ``monoidal'' assemblers
(assemblers equipped with an operation which is associative, commutative, and
unital up to a natural isomorphism) are taken to $E_\infty$-ring spectra.

\begin{maintheorem} \label{thm:maina} The $K$-theory functor
  $K: \Assemb \rto \Sp$ is monoidal and takes symmetric monoidal assemblers to
  $E_\infty$-ring spectra and symmetric monoidal morphisms of assemblers to
  ring homomorphisms on $K$-groups.
\end{maintheorem}

For more rigorous definitions and theorem statements, see
Section~\ref{sec:Kmonoidal}.  Moreover, we show in Theorem~\ref{thm:K1mult} that
the generators and relations given in \cite[Theorem B]{Z-ass-pi1} for
$K_1(\Var_k)$ interact in a natural manner with the multiplicative structure.

\subsection*{Organization}  This paper is targeted towards those interested in
the applications, including the derived $\zeta$-function.  In the service of
this, we front-load the applications, and leave the proofs of the structural
theorems for later sections.  Section~\ref{sec:review} contains a quick review
of assemblers and the results vital for an understanding of the applications.
Section~\ref{sec:ex} gives several examples of interest to this paper.
Section~\ref{sec:Gsets} analyzes non-permutative elements in $G$-sets, which
is used in Section~\ref{sec:zetamor} to detect non-permutative elements in
$K(\Var_k)$; Section~\ref{sec:zetamor} has all of the main applications of the
main theorem in this paper.  The last four sections cover the technical
underpinnings of Theorem~\ref{thm:maina}.  Section~\ref{sec:technical} gives a
run-down of the technical results necessary for the proofs.
Section~\ref{sec:monoidal} proves that the claimed product on assemblers
produces a symmetric monoidal structure.  Section~\ref{sec:Kmonoidal} proves
Theorem~\ref{thm:maina}.  Lastly, Section~\ref{sec:K1} proves that generators of
$K_1$ interact with the monoidal structure in the expected manner.

\subsection*{Acknowledgements} The author is grateful to Anna-Marie Bohmann,
David Corwin,  Brian Huang, Niles Johnson, and Angelica Osorno for helpful
conversations and for answering all sorts of questions.  The author would also
like to extend special thanks to Thomas Barnet-Lamb for his extreme patience
with all of the different variations of basic number theory questions that came
up during the writing of this paper.  

\section{A quick run-down of assemblers} \label{sec:review}

\begin{definition}[{\cite[Definition 2.4]{Z-Kth-ass}}] \label{def:ass}
  In a Grothendieck site $\C$ with an initial object, a \emph{covering family}
  is a family of morphisms which generates a covering sieve.  A family
  $\mathscr{F}$ is \emph{disjoint} if for any two morphisms $f:A \rto C$ and
  $g: B \rto C$ in $\mathscr{F}$, $A\times_C B$ exists an is equal to the
  initial object.

  An \emph{assembler} is a Grothendieck site $\C$, satisfying the following
  extra conditions:
  \begin{itemize}
  \item[(I)] $\C$ has an initial object $\initial$, and $\initial$ has an empty
    covering family.
  \item[(R)] For any two finite disjoint covering families of an object $A$,
    there is a common refinement which is itself finite and disjoint.
  \item[(M)] All morphisms in $\C$ are monic.
  \end{itemize}
  An assembler is \emph{closed} if the category $\C$ has all pullbacks.  Note
  that in this case, axiom (R) holds automatically.

  We generally assume that the initial object in $\C$ is unique, although this
  assumption does not affect any of the constructions in this paper.

  For an assembler $\C$, we write $\C^\circ$ for the full subcategory of
  noninitial objects of $\C$.

  A morphism of assemblers is a functor which preserves the initial objects and
  disjointness and which is continuous with respect to the topology.  Write
  $\Assemb$ for the category of assemblers, and $c\Assemb$ for the subcategory
  of closed assemblers and pullback-preserving morphisms of assemblers.
\end{definition}

The main examples of assemblers appearing in this paper are the following:

\begin{example}
  Let $G$ be a discrete group.  The assembler $\S^\Assemb_G$ has two objects,
  $\initial$ and $*$, with one morphism $\initial \rto *$.  In addition, $*$ has
  automorphism group $G$.  The topology is generated by the trivial covering
  families on $\initial$ and $*$, together with the empty covering family of
  $\initial$. When $G$ is trivial we omit it from the notation.
\end{example}

\begin{example}
  Let $G$ be a group.  The assembler $\Fin_G$ has as objects the finite
  $G$-sets, with morphism $G$-equivariant injections.  A family is a covering
  family if it is mutually surjective.  In other words, a family
  $\{f_i:A_i \rto A\}_{i\in I}$ is a covering family if
  $\bigcup_{i\in I} f_i(A_i) = A$.  (As before, when $G$ is trivial we omit it
  from the notation.)  When $G$ is profinite, the assembler $\AFSet_G$ is the
  assembler of \emph{almost-finite} $G$-sets: those $G$-sets $S$ such that $S^H$
  is finite for any closed subgroup $H \leq G$ and such that for all $x\in S$,
  the orbit $G \cdot x$ is finite.  Again, the morphisms are $G$-equivariant
  inclusions and covering families are mutually surjective.
\end{example}

\begin{example}
  Let $k$ be a field.  The assembler $\Var_k$ has as objects $k$-varieties
  (i.e. reduced separated schemes of finite type over $k$) and as morphisms
  locally closed immersions.  The topology is generated by the coverage
  consisting of families $\{f:Y \rcofib X, X \bs Y \rcofib X\}$, where $f$ is a
  closed immersion.  More generally, for a Noetherian scheme $S$ the assembler
  $\Var_S$ with objects varieties over $S$ is defined analogously to $\Var_k$.
\end{example}

Assemblers have a $K$-theory which classifies ``scissors congruence'' of the
objects of the assembler:

\begin{theorem}[{\cite[Theorem A]{Z-Kth-ass}}] \label{thm:K0}
  There exists a functor $K: \Assemb \rto \Sp$ from the category of assemblers
  to the category of spectra such that for any assembler $\C$, $\pi_0K(\C)$ is
  the free abelian group generated by objects of $\C$ modulo the relations
  \[[A] = \sum_{i\in I} [A_i] \qquad \hbox{for any finite disjoint covering
      family $\{A_i \rto A\}_{i\in I}$.}\]
\end{theorem}

\begin{definition}
  For an assembler $\C$, write $K_n(\C) \defeq \pi_n K(\C)$.
\end{definition}

Although generators and relations for $K_n(\C)$ for $n > 0$ akin to those given
in the above theorem are difficult to come by, there is a simple description of
those elements in $K_1(\C)$ which are of interest in the current context:
\begin{theorem}[Corollary \ref{cor:K1sp}] \label{thm:K1}
  Let $\C$ be an assembler.  
  Let $A$ be an object in $\C$, an let $\sigma\in \Aut(A)$.  Then the pair $[A,
  \sigma]$ represents an element of $K_1(\C)$.  These satisfy the following
  relations:
  \begin{itemize}
  \item For any finite disjoint covering family $\{f_i:A_i\rto A\}_{i\in I}$
    such that for each $i\in I$, there is a $\sigma_i\in \Aut(A_i)$ making the
    square
    \begin{diagram}
      { A & A \\ A_i & A_i \\};
      \arrowsquare{\sigma}{f_i}{f_i}{\sigma_i}
    \end{diagram}
    commute, 
    \[[A,\sigma] = \sum_{i\in I} [A_i, \sigma_i].\]
  \item If $\sigma, \sigma'\in \Aut(A)$ then
    \[[A,\sigma] + [A,\sigma'] = [A, \sigma\circ \sigma'].\]
  \end{itemize}
\end{theorem}
This theorem does not claim that these are the \emph{only} relations satisfied
by these elements, or that $K_1(\C)$ is generated by these elements.  However,
as these are the only elements of interest in this paper we restrict our
attention to this simpler statement.  For a more comprehensive analysis, see
\cite[Theorem B]{Z-ass-pi1}.

There is a monoidal structure on the category of closed
assemblers.\footnote{Although it should be possible to extend this construction
  to all assemblers, this would require more technical work which would distract
  from the main idea.  As all examples of interest to us are closed we focus on
  this subclass of examples.}  The intuition behind it comes from the following
example:

\begin{example} \label{ex:seg}
  Let $\Seg$ be the assembler whose objects are closed intervals in
  $\RR$ and whose morphisms are isometric injections.  Let $\mathbf{Rec}$ be the
  assembler whose objects are sets of the form $[a,a'] \times [b,b']$ in
  $\RR^2$, with morphisms isometric embeddings.  (Again, the topologies consist
  of the mutually surjective families.)  An object in $\mathbf{Rec}$ is a ``pair
  of objects'' in $\Seg$, and any decomposition of the elements of the
  pair produces a decomposition of the whole object:
  
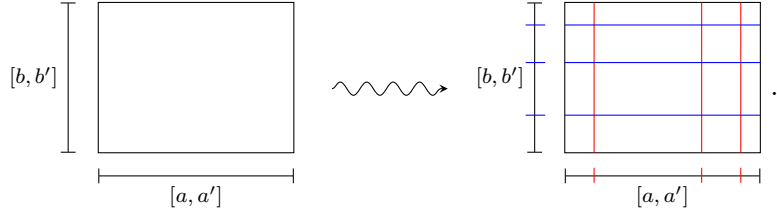
\begin{center}
    \begin{tikzpicture}[anchor=base,baseline,yshift=-2em,xscale=1.3]
      \draw [|-|,yshift=-0.8em] (0,0) to node[font=\scriptsize,below]  {$[a,a']$} (2,0);
      \draw [|-|,xshift=-0.8em] (0,0) to node[font=\scriptsize,left] {$[b,b']$} (0,2);
      \draw (0,0) rectangle (2,2);
    \end{tikzpicture}
    \setlen{4em}{\inlineArrow{->, decorate, decoration={snake}}}
    \begin{tikzpicture}[anchor=base,baseline,yshift=-2em,xscale=1.3]
      \draw [|-|,yshift=-0.8em] (0,0) to node[font=\scriptsize,below]  {$[a,a']$} (2,0);
      \draw [|-|,xshift=-0.8em] (0,0) to node[font=\scriptsize,left] {$[b,b']$} (0,2);
      \draw (0,0) rectangle (2,2);

      \draw[red] (0.3,-0.2) -- (0.3,-0.4) (1.4,-0.2) -- (1.4,-0.4) (1.8,-0.2) --
      (1.8,-0.4);
      \draw[red] (0.3,0) -- (0.3,2) (1.4,0) -- (1.4,2) (1.8,0) -- (1.8,2);
      \draw[blue] (-0.2,0.5) -- (-0.4,0.5) (-0.2,1.2) -- (-0.4,1.2) (-0.2,1.7)
      -- (-0.4,1.7);
      \draw[blue] (0,0.5) -- (2,0.5) (0,1.2) -- (2,1.2) (0,1.7) -- (2,1.7);
    \end{tikzpicture}
    .
  \end{center}
\end{example}

The idea of the monoidal structure on $\Assemb$ is to produce covering families
which are similarly ``gridded'' in the general setting.

\begin{definition}
  Let $\C$ and $\D$ be two closed assemblers.  The assembler $\C\sma \D$ has as
  underlying category the full subcategory of $\C\times \D$ consisting of those
  pairs $(C,D)$ where $C = \initial$ if and only if $D = \initial$.  The
  topology on this assembler is generated by the coverage consisting of those
  families
  \[\{(A_i,B_j) \rto (A,B)\}_{(i,j)\in I\times J}\]
  where both $\{A_i \rto A\}_{i\in I}$ and $\{B_j \rto B\}_{j\in J}$ are
  covering families in $\C$ and $\D$, respectively.
\end{definition}

This is not the usual topology on the product of sites.  The usual topology on
the product would have as covering families those families
$\{(A_i,B_i) \rto (A,B)\}_{i\in I}$ where $\{A_i \rto A\}_{i\in I}$ and
$\{B_i \rto B\}_{i\in I}$ are covering families in $\C$ and $\D$, respectively.
In particular, the family of shaded rectangles
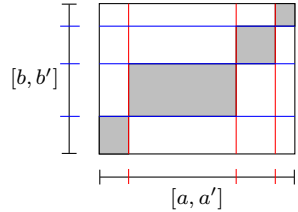
\begin{center}
\begin{tikzpicture}[anchor=base,baseline,yshift=-2em,xscale=1.3]
  \draw [|-|,yshift=-0.8em] (0,0) to node[font=\scriptsize,below]  {$[a,a']$} (2,0);
  \draw [|-|,xshift=-0.8em] (0,0) to node[font=\scriptsize,left] {$[b,b']$} (0,2);
  \draw (0,0) rectangle (2,2);

  \draw[fill=lightgray] (0,0) rectangle (0.3,0.5) rectangle (1.4,1.2) rectangle
  (1.8,1.7) rectangle (2,2);
 
  \draw[red] (0.3,-0.2) -- (0.3,-0.4) (1.4,-0.2) -- (1.4,-0.4) (1.8,-0.2) --
  (1.8,-0.4);
  \draw[red] (0.3,0) -- (0.3,2) (1.4,0) -- (1.4,2) (1.8,0) -- (1.8,2);
  \draw[blue] (-0.2,0.5) -- (-0.4,0.5) (-0.2,1.2) -- (-0.4,1.2) (-0.2,1.7)
  -- (-0.4,1.7);
  \draw[blue] (0,0.5) -- (2,0.5) (0,1.2) -- (2,1.2) (0,1.7) -- (2,1.7);
\end{tikzpicture}
\end{center}
gives a covering family under the standard topology, but not under the
topology in $\Seg \sma \Seg$.  In $\Seg \sma
\Seg$ all $16$ rectangles in the picture are necessary for a covering
family.

\begin{remark}
  This definition of the topology is interesting in that it allows us to produce
  a natural example of a Waldhausen category which does not satisfy the
  Saturation Axiom.  See Example~\ref{ex:nonsaturated}.
\end{remark}

The main technical result of this paper is the following:

\begin{theorem}[Lemma~\ref{lem:smasym},
  Section~\ref{sec:Kmonoidal}] \label{thm:monoidal} The structure
  $(c\Assemb,\sma, \S)$ is a symmetric monoidal structure on the category of
  closed assemblers.  The $K$-theory functor is monoidal and thus takes monoid
  objects to ring spectra.

  Given a symmetric monoidal assembler $\C$ (see
  Definition~\ref{def:symmonass}), $K(\C)$ is an $E_\infty$-ring spectrum.  A
  monoidal morphism of symmetric monoidal assemblers $\C \rto \D$ induces a ring
  homomorphism $K_*(\C) \rto K_*(\D)$.
\end{theorem}
Here, a ``symmetric monoidal assembler'' is a weakened form of a monoid object
in assemblers; it is directly analogous to the definition of a symmetric
monoidal category with the cartesian product of categories replaced by the
$\sma$-product of assemblers.

In particular, this theorem implies that if $\C$ is a symmetric monoidal
assembler then $K_*(\C)$ is a graded-commutative ring.  

\begin{proposition}
  Let $\C$ be a symmetric monoidal assembler with multiplication map $\mu$, and
  let $[A],[A']\in K_0(\C)$ and $[B, \sigma]\in K_1(\C)$.  Then
  \[[A][A'] = [\mu(A,A')] \qqand [A][B,\sigma] = [\mu(A,B), \mu(1_A, \sigma)].\]
\end{proposition}

\section{Examples} \label{sec:ex}

\subsection{Finite $G$-sets} \label{ex:finG}

\begin{notation}
  Let $G$ be a group.  Denote by $C_G$ a set of representatives for conjugacy
  classes of subgroups of $G$.
\end{notation}

Let $G$ be a finite group, and let $\Fin_G$ be the assembler whose objects are
finite $G$-sets, and whose morphisms are $G$-equivariant injections.  The
topology on the assembler is generated by mutually surjective families.

Every finite $G$-set is a disjoint union of its $G$-orbits.  Since all morphisms
in $\Fin_G$ are injective, the image of any $G$-orbit is an isomorphic
$G$-orbit.  By picking a point in an orbit, we see that any $G$-orbit is
isomorphic to $G/H$ for some subgroup $H$, and $G/H$ is isomorphic to $G/H'$ if
and only if $H$ and $H'$ are conjugate in $G$.  Let $\Fin_G^{[H]}$ be the full
subassembler of $\Fin_G$ whose objects are disjoint unions of $G$-orbits, each
of which is isomorphic to $G/H$.  Then
\[\Fin_G \cong \prod_{H\in C_G} \Fin_G^{[H]}.\]
Since $K$-theory of assemblers commutes with finite products, in order to
analyze the $K$-theory of $\Fin_G$ it is sufficient to understand the $K$-theory
of each $\Fin_G^{[H]}$.

Fix $H\in C_G$.  Let $S\in \Fin_G^{[H]}$ have exactly one $G$-orbit.  Then every
other object in $\Fin_G^{[H]}$ is isomorphic to a disjont union of copies of
$S$, and the automorphism group of $S$ is $W_GH$, the Weyl group of $H$ in $G$.
Thus by d\'evissage for assemblers \cite[Theorem B]{Z-Kth-ass}
$K(\Fin_G^{[H]}) \simeq K(\S^\Assemb_{W_GH}) \simeq \Sigma^\infty_+ B(W_GH)$.
Putting this together, we have
\begin{equation} \label{eq:KGSet}
  K(\Fin_G) \simeq \prod_{H\in C_G} \Sigma^\infty_+ B(W_GH).
\end{equation}

The Cartesian product of finite $G$-sets produces a functor
$\mu:\Fin_G \sma \Fin_G \rto \Fin_G$.  This preserves the initial object and
disjointness by definition, so in order to check that it is a morphism of
assemblers it suffices to show that it takes covering families to covering
families; in particular, it suffices to check that for covering families
$\{f_i:A_i \rto A\}_{i\in I}$ and $\{g_j:B_j \rto B\}_{j\in J}$, the family
\[\{(f_i,g_j): A_i \times B_j \rto A\times B\}_{(i,j)\in I\times J}\] is a
covering family.  This is true because for any $(a,b)\in A\times B$, if we
choose $i\in I$ such that $a\in \im f_i$ and $j\in J$ such that $b\in \im g_j$,
$(a,b)$ will be in the image of $(f_i,g_j)$, as desired.

By definition,
\[K_0(\Fin_G) \cong A(G),\]
the Burnside ring of $G$; the monoidal structure above extends the ring
structure of the Burnside ring to the higher $K$-groups.  

To finish up this section we want to make a couple of observations which will be
useful when we discuss almost-finite sets in Section~\ref{sec:afsets}.  Suppose
that $G = \Z/p^n$.  We define $c_i: \Fin_{\Z/p^n} \rto \Fin$ by
\[c_i(S) \defeq \big\{s\in S \,\big|\, |G\cdot s| \leq p^i\big\}.\]
The map $\prod_{i=0}^n c_i: \Fin_G \rto \prod_{i=0}^n \Fin$ is the ``ghost
coordinate'' map: on $K_0$ it takes a $G$-set to coordinates that add/multiply
coordinatewise, and all information about the relative sizes of the $G$-orbits
can be recovered from them.

This can also be coordinatized in an alternate manner, using ``Burnside
coordinates.''  Define an ``orbit counting map'' $b_i: \Fin_{\Z/p^n} \rto \Fin$
by
\[b_i(S) \defeq \big\{s\in S \,\big|\, |G\cdot s| = p^i\big\}_G.\] Note that the
relationship between $|b_i(S)|$ and $|c_i(S)|$ looks closely related to the
relationship between the Witt coordinates and the ghost coordinates:
\[|c_i(S)| = \sum_{j=0}^i p^j |b_j(S)|.\] We can then define
$b = \prod_{i=0}^n b_i: \Fin_{\Z/p^n} \rto \Fin^n$; note that this is an
(additive) isomorphism on $K_0$.\footnote{The terminology ``$b$'' for the
  orbit counting map may seem arbitrary, but it is chosen to agree with the
  ``$b$-coordinates'' defined for the isomorphism $\tau$ from the standard
  presentation of the Witt ring to the Burnside ring described in
  \cite[p7]{dresssiebeneicher89}.}

We get the following commutative diagram:
\begin{diagram}[4em]
  { \Fin^n \sma \Fin^n & \Fin_{\Z/p^n} \sma
    \Fin_{\Z/p^n} & \Fin^n \sma \Fin^n \\
     \Fin^n & \Fin_{\Z/p^n} & \Fin^n \\};
  \to{1-2}{1-1}_{b\sma b} \to{1-2}{1-3}^{c\sma c}
  \to{2-2}{2-1}_b \to{2-2}{2-3}^c 
  \to{1-2}{2-2}^\times \to{1-3}{2-3}^\times 
  \diagArrow{densely dotted, ->}{1-1}{2-1}
\end{diagram}
What must the dotted map be in order to make the diagram commute?  On each pair
of coordinates, the dotted map counts the number of each type of orbit which
appears in the Cartesian product of orbits.  Since $b$ forgets all $G$-action
information, this is functorial and it does not matter which representatives are
taken when the dotted map is defined.  Thus the dotted map exists.  All of the
horizontal maps in ths diagram are isomorphisms on $K_0$, and the conversion
between ``ghost coordinates'' and ``Burnside coordinates'' is the conversion
between the right-hand side of the diagram and the left-hand side of the
diagram.

\subsection{Almost-finite $G$-sets} \label{sec:afsets} Now let $G$ be a
profinite group, and let $S$ be a $G$-set.  $S$ is \emph{almost-finite} if for
all open subgroups $U$ of $G$, $S^U$ is finite, and if for all $x\in S$, the
orbit $G\cdot x$ is finite.  For an in-depth discussion of almost-finite sets
(and spaces), see \cite{dresssiebeneicher88}.  We define $\AFSet_G$ to be the
assembler whose objects are almost-finite $G$-sets and whose morphisms are
$G$-equivariant inclusions.  The topology on $\AFSet_G$ is given by the mutually
surjective covering families.

Although we would like to use d\'evissage for assemblers to compute the
$K$-theory of $\AFSet_G$, this is not directly possible, since an almost-finite
set can be the union of infinitely many different $G$-orbits.  Thus we need to
be a little bit more clever.  For an open subgroup $U$ of $G$, let $\C_{G/U}$ be
the full subcategory of $\SC(\AFSet_G)$ containing only those $G$-sets which are
unions of orbits isomorphic to $G/U$.  Each such set is a finite disjoint union
of copies of $G/U$, and thus $\C_{G/U} \cong \SC(\S_{W_GU})$.  Thus
\[\SC(\AFSet_G) \simeq \prod_{\substack{\mathrm{conj.\ class}\\U \leq G}}
\C_{G/U} \cong \prod_{\substack{\mathrm{conj.\ class}\\U \leq G}}
\SC(\S_{W_GU}).\]  For any conjugacy class $U$, the functor projecting to the
$U$-coordinate is a morphism of assemblers, so induces a map on $K$-theory; in
particular after applying $K$-theory there exists a map
\begin{equation} \label{eq:decomp}
  K(\AFSet_G) \rto^\psi \prod_{U\in C_G}
  \Sigma_+^\infty B(W_GU).
\end{equation}
Write $\psi_U: K(\AFSet_G) \rto \Sigma_+^\infty BW_GU$ for the projection onto
the $U$-th coordinate.

The ring $K_0(\AFSet_{G})$ is exactly the Burnside ring of $G$, so $\psi$
induces an isomorphism on $K_0$.  By \cite[Corollary 1 and Theorem
2.12.7]{dresssiebeneicher89}, the Burnside ring is isomorphic to the big Witt
ring, where the coordiates can be considered to be the ``orbit counting'' maps
(analogous to $b_i$ in the previous section).  Addition is represented by the
usual disjoint union of sets, while multiplication is given by Cartesian product
of almost-finite sets, with the unit the singleton set with the trivial
$G$-action.  Note that this multiplication, while induced by a morphism of
assemblers $\AFSet_G \sma \AFSet_G \rto \AFSet_G$ is not induced by
multiplications on components $\S_{G/U} \sma \S_{G/U} \rto \S_{G/U}$,
illustrating that the decomposition above is not compatible with the
multiplicative structure.  (More concretely: the product of two $G/U$-orbits
decomposes into separate orbits; it cannot be modeled by a product of singleton
sets.)  This is where the interesting multiplicative properties of Witt vectors
come from.

To illustrate this last observation, consider the case $G = \Z_p$ and the maps
$b$ and $c$ defined in the previous section.  There is an analogous commutative
diagram
\begin{diagram}
  {\Fin^{\NN} \sma \Fin^{\NN} & \Fin_{\Z_p} \sma \Fin_{\Z_p} &
    \Fin^\NN \sma \Fin^\NN \\
    \Fin^{\NN} & \Fin_{\Z_p} & \Fin^{\NN} \\};
  \to{1-2}{1-1}_{b\sma b} \to{1-2}{1-3}^{c\sma c}
  \to{2-2}{2-1}_b \to{2-2}{2-3}^c
  \to{1-3}{2-3}^\times \to{1-2}{2-2}^\times \diagArrow{densely dotted,->}{1-1}{2-1}
\end{diagram}
Again, the multiplication induced on the $b$-coordinates is not a simple
coordinate-wise one, as the product of two orbits of size $p^i$ (for example) is
not a single orbit of size $p^i$.  The formula for the relationship between the
$c$-coordinates and the $b$-coordinates is exactly that between the ghost
coordinates and the Burnside ring.  If we compose with the canonical isomorphism
between the Burnside ring and the Witt ring, this becomes the standard
coordinate transformation between the Witt ring and the ghost coordinates.  For
more details, see \cite{dresssiebeneicher88}.

Moreover, these two separate perspectives illustrate where many of the strange
multiplicative properties of the Witt vectors come from.  The Witt coordinates
are closely related to the Burnside coordinates; however, under multiplication,
the Burnside coordinates do not simply multiply.  Consider that the product of a
single orbit of size $p^i$ and an orbit of size $p^j$ (with, WLOG, $j \geq i$)
is $p^i$ orbits of size $p^j$; thus the product of a vector with a single $1$ in
position $i$ with a vector with a single $1$ in position $j$ is a vector with a
$p^i$ in position $j$.  

\begin{remark}
  In the discussion above it is tempting to use a result showing that $K$-theory
  commutes with infinite products, such as \cite{carlsson95} or
  \cite{kasprowskiwinges20}, to conclude that $\psi$ is an isomorphism.
  Unfortunately, the $K$-theory of assemblers cannot directly use either of
  these results---the former assumes Waldhausen categories with cylinder
  functors (which $\SC(\C)$ does not have) and the latter assumes exact
  categories---and thus these results are not accessible to us at this point.

  As we expect that $K$-theory of assemblers does commute with infinite
  products, we conjecture that the map $\psi$ is actually a weak equivalence.
\end{remark}

\subsection{Varieties} \label{sec:varieties} Let $\Var_k$ be the assembler of
varieties over a base field $k$.  The morphisms in the assembler are locally
closed embeddings; the topology is generated by the coverage
$\{Y \rcofib X, X \bs Y \rto X\}$, where $Y \rcofib X$ is a closed embedding.
This example is discussed in detail in \cite[Section 5.1]{Z-Kth-ass}.  The fiber
product (over $\Spec k$) of varieties induces a product structure on $\Var_k$,
with $\Spec k$ as the unit.  This produces an $E_\infty$-ring structure on
$K(\Var_k)$ which on $K_0$ gives the Grothendieck ring of varieties.

If we would like to work over a base Noetherian scheme $S$, instead of a base
field, an analogous construction works with the ring structure given by taking
the fiber product over $S$ (in which case $S$ is the unit).  (Again, we assume
that varieties over $S$ are reduced separated schemes of finite type over $S$.)
It is important to note that the technical condition in the following lemma
requires \textbf{equalities} and not \textbf{canonical isomorphisms}; the
desired formulas always hold up to canonical isomorphism.
\begin{lemma}
  Let $f:T \rto S$ be a morphism of schemes.  Base change along $f$ induces a
  symmetric monoidal morphism of assemblers $\Var_S \rto \Var_T$.  In
  particular, $K_*(\Var_S) \rto K_*(\Var_T)$ is ring homomorphism.
\end{lemma}

\begin{proof}
  Base change is a morphism of assemblers because it preserves disjointness and
  covering families in the generating coverage. 
  To check that it is a symmetric monoidal map it suffices to check that the
  following diagrams commute up to natural isomorphism:
  \[
    \begin{inline-diagram}[3em]
      { \Var_S \sma \Var_S & \Var_T \sma \Var_T \\
        \Var_S & \Var_T \\};
      \arrowsquare{f'\sma f'}{\mu}{\mu}{f'}
    \end{inline-diagram}
    \qqand
    \begin{inline-diagram}
      { \S & \Var_S \\
        & \Var_T. \\};
      \to{1-1}{1-2}^\eta \to{1-1}{2-2}_\eta \to{1-2}{2-2}^{f'}
    \end{inline-diagram}
  \]
  Since $f'$ maps $S$ to $T$, the right-hand diagram commutes.  Since both base
  change and the multiplication are via fiber products, the left-hand diagram
  commutes up to natural isomorphism.
\end{proof}

\begin{remark}
  It may seem that this lemma is overly-complicated: the given formulas always
  hold up to unique isomorphism, and for most purposes this is sufficient.
  However, this is \emph{not} sufficient in the current case: in order to define
  a monoid object in a category, the given diagrams must commute \emph{exactly},
  not simply up to unique isomorphism.  This is, in fact, precisely the problem
  that $\infty$-categories are designed to address: situations where coherence
  issues occur because of imprecise commutativity.  In the current situation it
  ought to be the case that the weaker commutativity should be sufficient, and
  that the map $K(\Var_S) \rto K(\Var_T)$ should be $E_\infty$ regardless of
  which model is taking.  However, proving this would require keeping track of
  the $E_\infty$-operad structure through the $K$-theoretic machinery, which can
  be quite complicated.  (See for example \cite{elmendorfmandell} for a
  description of this; note that their example is not sufficient for the current
  application, and would need to be weakened further, as the current situation
  would not be bipermutative.)  We therefore take the alternate track of
  focusing on the special case of interest to us.
\end{remark}

\begin{example}
  Let $A$ and $B$ be rings, with a map $f: \Spec B \rto \Spec A$.  Write
  $\tilde \Var_A$ for the full subassembler of $\Var_{\Spec A}$ of reduced
  separated affine schemes of finite type over $\Spec A$.  Since all varieties
  have a finite disjoint covering family by affines, the inclusion
  $\tilde \Var_A \rto \Var_{\Spec A}$ induces an equivalence on $K$-theory (by
  d\'evissage, \cite[Theorem B]{Z-Kth-ass}).  Define the assembler
  $\tilde \Var_B$ analogously.

  The multiplication and the base change to working over $B$ on $\tilde \Var_A$
  can be modeled by the tensor product of $A$-algebras, and it suffices to
  define a tensor product for which the formulas $A\otimes_A B = B$ and
  \[(R\otimes_A R')\otimes_A B = (R\otimes_A B)\otimes_B (R' \otimes_A B)\]
  hold.  (Note, again, that these are equalities, and not isomorphisms; these
  will of course always be canonically isomorphic.)

  For example, here is a method for building such a model.  Pick a tensor
  product functor on $\tilde \Var_A$.  For every $A$-algebra $R$ there is a
  $B$-algebra $R\otimes_AB$ which is generated by pairs $(r,b)$.  Define a
  tensor product on $\tilde \Var_B$ by defining
  $(R\otimes_A B) \otimes_B (R'\otimes_A B)$ to be generated by classes of
  triples $(r,r',b)$ modulo the necessary relations.  (This differs
  from the usual definition: usually we would take pairs of pairs
  $((r,b),(r',b'))$ and define the appropriate relations on those.)  Thus, on
  the subcategory of those modules which are in the image of base change from
  $\tilde\Var_A$, we impose the formula by definition.
\end{example}

In a more combinatorial example, we can construct a derived $\zeta$-function.  
Let $k$ be a finite field, $\Var_k$ be the assembler of varieties over $k$, and
$\AFSet_{\hat \Z}$ be the assembler of almost-finite $\hat \Z$-sets.  There is a
morphism of assemblers
\begin{equation} \label{eq:zeta}
  \zeta: \Var_k \rto \AFSet_{\hat Z} \qquad \hbox{given by } X \rgoesto X(\bar
  k).
\end{equation}

\begin{lemma} \label{lem:zetamon} The morphism of assemblers $\zeta$ is a
  symmetric monoidal morphism, in the sense that the following diagrams commute
  up to natural isomorphism:
  \[
    \begin{inline-diagram}
      {\Var_k \sma \Var_k & \AFSet_{\hat \Z} \sma \AFSet_{\hat \Z} \\
        \Var_k & \AFSet_{\hat \Z}\\};
      \arrowsquare{\zeta \sma \zeta}{\mu}{\mu}{\zeta}
    \end{inline-diagram}
    \qqand
    \begin{inline-diagram}
      {\S^\Assemb & \Var_k \\ & \AFSet_{\hat \Z}.\\};
      \to{1-1}{1-2}^\eta \to{1-1}{2-2}_\eta \to{1-2}{2-2}^\zeta
    \end{inline-diagram}\]
  In particular, $K_*(\zeta)$ is a ring homomorphism.
\end{lemma}

\section{Detecting non-permutative elements in $G$-sets} \label{sec:Gsets}

In this section we discuss how to use the theory developed above to detect
nontrivial elements in $K_1$ of an assembler.  The particular elements we are
concerned with are the ``non-permutative'' elements:

\begin{definition}
  Let $E$ be a connective $E_\infty$-ring spectrum with unit map $\S \rto E$.
  An element in $\pi_nE$ is \emph{$0$-dimensional} if it is in the image of
  $\pi_n\S \rto \pi_nE$.  An element in $\pi_nE$ is \emph{permutative} if it is
  in the image of the map
  \[\pi_0E \otimes \pi_n\S \rto \pi_0E \otimes \pi_n E \rto \pi_n E.\]
  If an element is not permutative then it is \emph{non-permutative}.
\end{definition}

In \cite{CWZ-zeta} the authors showed that there exist elements in $K_n(\Var_k)$
which are not $0$-dimensional; however, the question of whether there exist
non-permutative elements was left open.  The goal of the rest of this paper is
to show that such elements exist in $K(\Var_k)$.  The existence of
non-permutative elements in the higher homotopy of $E$ demonstrates that $E$ is
not uniquely determined by $\pi_0E$ and the higher homotopy groups of $\S$.  In
particular, we will use it to demonstrate that the higher $K$-groups of $\Var_k$
contain nontrivial information about the geometry of varieties.

The important theorem for determining that certain elements are non-permutative
is the following theorem, whose proof is put off until later:
\begin{theorem} \label{thm:K1mult} For a closed symmetric monoidal assembler
  $\C$, $[X]\in K_0(\C)$, and $[A,\tau]\in K_1(\C)$,
  \[[X] [A,\tau] = [X\times A, 1 \times \tau] \in K_1(\C).\]
\end{theorem}
This is a special case of Theorem~\ref{thm:K1mult-real}.

\subsection{Finite $G$-sets}

First, consider the simple case of finite $G$-sets.  From (\ref{eq:KGSet}) it
follows that
\[K_1(\Fin_G) \cong \bigoplus_{H\in C_G} \Z/2 \oplus (W_GH)^\ab.\] 

By Theorem~\ref{thm:K1mult}, the product of an element in $K_0(\Fin_G)$,
represented as $[A]$ for some finite $G$-set $A$, and the nontrivial element in
$K_1(\Fin)$, represented by $[\{1,2\}, \tau]$ is represented by the element
$[A \amalg A, \tau]$, where $\tau$ swaps the two copies of $A$.  In particular,
neither of the morphisms in this pair has a nontrivial action on any $G$-orbit
of $A$; thus this lands in the subgroup $\bigoplus_{H\in C_G} \Z/2 \oplus 0$.
Conversely, any term in this subgroup is represented, since
$[G/H \amalg G/H, \tau]$ represents the element with a $1$ in the $H$-coordinate
and $0$'s everywhere else.  Thus the non-permutative elements in the group are
exactly those that are outside this subgroup.

\subsection{Almost-finite $G$-sets}

The analysis of the previous subsection can be extended to almost-finite
$G$-sets, although there is again the additional difficulty that we do not know
whether $K$-theory commutes with infinite products.  However, we can show the
following:
\begin{proposition}
  Recall the map $\psi$ from (\ref{eq:decomp}).  All permutative elements
  $\alpha\in K_1(\AFSet_G)$ have
  \[\psi_*\alpha \in \prod_{H \in C_G} Z/2 \oplus 0 \subset \prod_{H\in C_G} \Z/2 \oplus
    (W_GH)^\ab.\]
  Moreover, the subgroup of permutative elements surjects onto this subgroup.
\end{proposition}

\begin{proof}
  The proof that the image of any permutative element lies in this group is
  identical to the previous subsection.  That every element in this subgroup is
  represented can be seen by taking an element in the subgroup, which is of the
  form $\{\epsilon_H\}_{H\in C_G}$ with $\epsilon_H = \pm 1$.  Consider the
  almost-finite $G$-set
  $A \defeq \coprod_{\substack{H\in C_G \\ \epsilon_H = 1}} G/H$.  Then
  $[A\amalg A, \tau]$ exactly represents $\{\epsilon_H\}_{H\in C_G}$.
\end{proof}

\section{A combinatorial derived zeta function}   \label{sec:zetamor}

Let $k$ be a finite field, and let $L_n$ be the unique extension of $k$ of
degree $n$.  The local zeta function of a variety $X$ over $k$ is defined by
\[Z(X,t) \defeq \exp \sum_{n \geq 1}
  |X(L_n)|\frac{t^n}{n}.\] There are two other classical ways of writing this
function.  We can think of the function $Z(-,t)$ as taking a variety to a power
series in $t$ with integer coefficients.  (We know that it will have integer
coefficients because we can rewrite the expression above as
$\sum_{n \geq 0} |(\Sym^nX)(k)|t^n$.)  We can see by analyzing the expression
for $Z(X,t)$ that, for a closed embedding $Y \rcofib X$,
\[Z(X,t) = Z(Y,t) Z(X\bs Y,t).\]
Thus $Z(-,t)$ is actually a homomorphism $K_0(\Var_k) \rto (1+t\Z\llbracket
t\rrbracket,\times)$.  If we consider the codomain as the big Witt ring with the
multiplication of Witt vectors, $Z(-,t)$ is a ring homomorphism.

By the discussion in Section~\ref{ex:finG}, the big Witt ring is
$K_0(\AFSet_{\hat \Z})$, and, indeed, we can see that all of the data necessary
to construct $Z(X,t)$ is contained in the $\hat\Z$-structure of $X(\bar k)$.
Thus the morphism of assemblers $\zeta$ in (\ref{eq:zeta})
gives the desired derived map
\[K(\zeta): K(\Var_k) \rto K(\AFSet_{\hat \Z}).\] Since $\zeta$ is a morphism of
monoidal objects, this map induces a map of rings.  Moreover, because this map
is compatible with the unit maps it takes permutative elements to permutative
elements.  In order to demonstrate that an element in $K_1(\Var_k)$ is
non-permutative it therefore suffices to show that it maps to a non-permutative
element in $K_1(\AFSet_{\hat \Z})$,

The map $\psi$ defined in  Section~\ref{sec:afsets} induces a map 
\begin{equation} \label{eq:zeta-decomp}
  K_1(\AFSet_{\hat\Z}) \rto \prod_{n \geq 1} K_1(\S^\Assemb_{\Z/n}) \cong
  \prod_{n \geq 1} \Z/2\oplus \Z/n.
\end{equation}
We write an element in the codomain as
$\prod_{n \geq 1}(\pm 1, g)$.  In each pair we call the first coordinate the
\emph{external} coordinate and the second the \emph{internal} coordinate;
these correspond to the sign of the permutation of $*$'s in the objects of
$\SC(\S^\Assemb_{\Z/n})$ and the action of $\Z/n$, respectively.
\begin{lemma} \label{lem:mult-of-eta}
  Let $[X]\in K_0(\Var_k)$, and let $\eta\in K_1(\Fin)$ be the nonzero element,
  and write $L_j$ for the extension of $k$ of degree $j$.  Let
  \[X_n = \{x\in X(L_n)\,|\, x\notin X(L_m),\ m < n\}.\]
  Then
  \[\psi_* \circ \zeta_*([X]\eta) = \prod_{n \geq 1} ((-1)^{|X_n|/n},0).\]
\end{lemma}

\begin{proof}
  Write $\psi_n$ for the composition of $\psi_*$ and the projection onto the
  coordinate indexed by $n$.  Write $\Fin^{\Z/n}$ for the assembler of finite
  sets with free $\Z/n$-action; then
  $K(\Fin^{\Z/n}) \simeq \Sigma_+^\infty B\Z/n$, and the map
  $\Var_k \rto \Fin^{\Z/n}$ mapping $X$ to $X_n$ gives the $n$-th coordinate of
  the map $\psi_*\zeta_*$.  There is a map of assemblers $\Fin \rto \Fin^{\Z/n}$
  defined by $S \rgoesto S\times \Z/n$, with $\Z/n$ acting on the second
  coordinate.  

  It suffices to show that for all $n$ there exists a
  morphism $\sigma_n$ which makes the following square commute:
  \begin{diagram}
    { K_1(\Var_k)  & K_1(\Fin^{\Z/n}) \\
      K_0(\Var_k) \otimes K_1(\S^\Assemb) & K_1(\Fin). \\};
    \to{1-1}{1-2}^{\psi_n} \to{2-1}{2-2}^{\sigma_n}
    \to{2-1}{1-1}_{\mu} \cofib{2-2}{1-2}
  \end{diagram}
  Pick a generator of $K_0(\Var_k)$, represented by a variety $[X]$ and the
  generator $[*\amalg *, \mathrm{swap}]$ of $K_1(\S^\Assemb)$.  The map $\mu$
  then maps this generator to the generator $[X \amalg X, \mathrm{swap}]$; the
  map $\psi_n$ maps this to the generator
  \[\big[X_n \amalg X_n, \mathrm{swap}\big].\]
  Note that $X_n \cong (X_n)_{\Frob} \times \Z/n$. Thus if we define
  \[\sigma_n([X]\otimes \eta) \defeq \big[(X_n)_{\Frob} \amalg (X_n)_{\Frob},
    \mathrm{swap}\big]\]
  the diagram commutes, as desired.  Moreover, the sign of the induced
  permutation is exactly the sign of swapping $|X_n|/n$ orbits, giving the
  desired formula.
\end{proof}
Since $K_*(\zeta)$ is a ring homomorphism, the image of a permutative element in
$K_*(\Var_k)$ is a permutative element in $K_*(\AFSet_{\hat\Z})$.  In
particular, permutative elements have all internal coordinates equal to $0$.  In
order to find a nonpermutative element it suffices to find an element which has
a nonzero internal coordinate.

By \cite[Theorem B]{Z-ass-pi1}, any automorphism of a variety represents an
element of $K_1(\Var_k)$.  The functor $\zeta$ takes this data to a $\hat\Z$-set
$X(\bar k)$ together with a $\hat\Z$-equivariant permutation; projecting onto
the $\Z/2$-coordinate in (\ref{eq:zeta-decomp}) induces a map
\[\psi_2: K_1(\Var_k) \rto K_1(\S^\Assemb_{\Z/2}) \cong \Z/2\oplus \Z/2.\]
Here, the first coordinate is the external coordinate, and the second is the
internal coordinate.  As before, we write the external coordinate
multiplicatively and the internal coordinate additively.

\begin{definition}
  Let $S$ be a finite set equipped with an action of $\Z/m\oplus \Z/n$.  Suppose
  that this action is free when restricted to both $\Z/m\oplus 1$ and
  $1\oplus \Z/n$.  For a point $x\in S$ write $[x]$ for the orbit of $x$ under
  the $\Z/m\oplus \Z/n$ action.  This orbit has \emph{type $(d,a)$} if
  $(d,0) \cdot x = (0,a) \cdot x$ for $0 \leq a < n$ and $d$ is the minimal
  positive integer for which such an integer $a$ exists.

  Define $\# S_{(d,a)}$ for the number of orbits of type $(d,a)$.

  Let $X$ be a variety over $\FF_q$ equipped with an automorphism $\varphi$
  which acts on $X_n$ freely with order $m$.  We consider $X_n$ to be equipped
  with action of $\Z/m\oplus \Z/n$ by having $(1,0)$ act by $\varphi$ and
  $(0,1)$ act by Frobenius.  Define $\#X^{\varphi,n}_{(d,a)}$ to be the number of orbits of
  type $(d,a)$ in $X_n$; when $n$ and $\varphi$ are clear from context we omit
  them.  
\end{definition}

If there exists an orbit of type $(d,a)$ in $S$ then it is necessarily the case
that $d |m$ and that $\frac md = \frac{n}{(n,a)}$.  In particular, $d = m$ if and only
if $a = 0$.  

We can use this to describe the image of an element $[X,\varphi]\in K_1(\Var_k)$
under $\psi_n$ explicitly in terms of the actions on orbits.

\begin{proposition} \label{prop:stable-unstable}
  Let $X$ be a variety over $k$ and let $\varphi$ be an automorphism of $X$;
  suppose that $\varphi$ acts freely on $X_n$ with order $m$.  For each pair
  $(d,a)$ write $\#X_{(d,a)}$ be the number of orbits of $X_n$ of type $(d,a)$.
  Then 
  \[\psi_n([X,\varphi]) = \sum_{(d,a)} (\#X_{(d,a)})((-1)^{d+1}, a).\]
\end{proposition}

\begin{proof}
  The element in $K_1(\Fin^{\Z/n})$ is the sum of its actions on orbits, so it
  suffices to consider a single orbit $[x]$ of type $(d,a)$.  This orbit
  consists of $d$ disjoint Frobenius orbits.  If we write $x_i = \varphi^i\cdot
  x$ for $i = 0,\ldots,d-1$ we can think of the $d$ Frobenius orbits as the sets
  $\{x_i,\Frob\cdot x_i, \ldots, \Frob^{n-1}\cdot x_i\}$ then the action of
  $\varphi$ takes the $i$-th orbit to the $i+1$-st orbit with no $\Z/n$-action
  for all $0 \leq i < d-1$; however, when $i = d-1$ it takes the $i$-th orbit to
  the $0$-th orbit with an additional $\Z/n$ action by $a$.  We can thus write
  it as a cyclic permutation of $d$ orbits, followed by a twist by $a$ on a
  single orbit.  Thus the representative in $K_1(\Fin^{\Z/n})$ is
  \[((-1)^{d+1},0) + (1,a) = ((-1)^{d+1}, a),\]
  as desired.

  Summing over all orbits gives the desired formula.  
\end{proof}

Proposition~\ref{prop:stable-unstable} and Lemma~\ref{lem:mult-of-eta} can be
used to find non-permutative elements in $K_1(\Var_{\FF_q})$.  Our first result
is a special case of the above result.

\begin{proposition} \label{prop:su-special} Fix an integer $n > 1$.  Let
  $\lambda \in k^\times$ have order $m$, and define $P_1 = \# (\PP^1)_{(m,0)}$
  and for any other $d | (n,m)$ let $P_d = \# (\PP^1)_{(m/d,n/d)}$; for
  all other $d$ we define $P_d=0$.  Let $\phi$ be the Euler $\phi$-function.
  Then
  \[\psi_n([\PP^1,\lambda x]) = 
      \bigg( (-1)^{P_1(m+1) + P_2(m/2+1)}, \sum_{d|(n,m)}
      P_{d} \phi(d) \frac{n}{d}\bigg),
  \]
  where $\phi$ is the Euler $\phi$-function.  In particular, if $(n,m) = 1$ then
  the second coordinate is $0$.  
\end{proposition}

\begin{proof}
  From the formula in Proposition~\ref{prop:stable-unstable}, and writing $X =
  \PP^1$ for conciseness,
  \[\psi_n([\PP^1,\lambda x]) = \sum_{(d,a)} (\#X_{(d,a)}) ((-1)^{d+1},a).\]
  Suppose that $d,a,a'$ are such that both $X_{(d,a)}$ and $X_{(d,a')}$ are
  nonempty.  The necessary conditions on $a$ and $a'$ ensure that there exists a
  constant $c$ and two primitive roots $g,g'$ such that
  \[\lambda^a = g^c \qqand \lambda^{a'} = (g')^c.\]
  With these primitive roots we can construct a bijection between $X_{(d,a)}$
  and $X_{(d,a')}$ in the following manner. Given an orbit $[x]\in X_{(d,a)}$,
  by definition
  \[x^{q^d} = \lambda^a x \Leftrightarrow x^{q^d-1} = \lambda^a.\] Writing
  $x = g^y$ , such a point $x$ is a solution $y$ to the equation
  \[y(q^d-1) \equiv ca \pmod{q^n-1}.\] In particular, this point $x$ corresponds
  to a point $x' \defeq (g')^y$.  This gives a function
  $X_{(d,a)} \rto X_{(d,a')}$; the inverse is given by the reverse choice of
  primitive roots.  This bijection shows that for any two $a$ and $a'$
  satisfying $\frac md = \frac n{(n,a)}$ it is the case that
  $\#X_{(d,a)} = \# X_{(d,a')}$.  Thus in the sum above we can choose
  $a = (n,a) = \frac n{m/d}$ for every $d$ and multiply by the number of choices
  of $a$, which is exactly $\phi(m/d)$.  Moreover, such a choice is possible
  only if $(m/d) \mathrel{|} n$.  Thus the sum can be rewritten as 
  \[\psi_n([\PP^1,\lambda x])  =
    \sum_{\substack{d|m \\ m/d | n}} \phi\left(\frac md\right) \# X_{(d,\frac
      n{m/d})} \left((-1)^{d+1}, \frac md\right) = \sum_{d | (n,m)} \phi(d) \#
    X_{(m/d,n/d)} ((-1)^{m/d+1}, d).\] Since $\phi(d)$ is even unless $d = 1$ or
  $2$, the first coordinate of each summand will almost always be $1$.  Taking
  this into account gives the desired formula.
\end{proof}

Using this we can do a a complete analysis of the image under $\psi_n$ of
$[\PP^1,1/x] = [\PP^1, -x]$, in order to illustrate both the benefits and the
drawbacks of the approach.

\begin{corollary} \label{cor:calc} Let $k = \FF_q$.  Writing $n = 2^m n'$ with
  $n'$ odd,
  \begin{equation} \label{eq:form}
    \psi_n([\PP^1,-x]) =
    \begin{cases}
      \left((-1)^{\frac{q-1}{2}}, 0\right) \caseif n = 1, \\
      \left((-1)^{\frac{q-1}{2}},\frac{q-1}{2}\right) \caseif n=2, \\
      (1,0) \caseotherwise.
    \end{cases}
  \end{equation}
  In particular, if $q \equiv 3 \pmod 4$ then $[\PP^1,-x] = [\PP^1, 1/x]$ is
  non-permutative.
\end{corollary}

\begin{proof}
  When $n = 1$ Galois orbits are trivial, and thus the internal coordinate is
  $0$.  Thus the question becomes to compute the sign of the permutation that
  $1/x$ induces on $\FF_q\bs\{0\}$.  There are two fixed points ($\pm 1$) and
  the rest are paired up into transpositions, so the sign is
  $(-1)^{\frac{q-1}{2}}$, as desired.
  
  For $n>1$ we use the formula in Proposition~\ref{prop:su-special}.  In this
  the formula simplifies to:
  \[\psi_n([\PP^1,-x]) = \begin{cases}
      \left( (-1)^{P_1},  P_{2} \frac n2 \right) \caseif n \hbox{ is even} \\
      ((-1)^{P_1},0) \caseotherwise.
    \end{cases}\]
  In particular, the result only depends on the parities of $P_1 = \# X_{(2,0)}$
  and $P_2 = \# X_{(1,\frac n2)}$.

  Before we begin the more complicated cases, some notation for the rest of the
  proof.  For integers $a$ and $b$, $M_a(b)$ is the number of aperiodic necklaces
  of length $b$ with beads of $a$ colors.  The function $\mu(m)$ is the Mobius
  function, which is $0$ if $m$ is not squarefree, and otherwise is $-1$ to the
  power of the number of distinct prime factors of $m$.  The symbol
  $\delta_{ij}$ is the Kronecker delta function.  The important facts about
  these to know are that
  \[M_a(b) = \frac 1b \sum_{d|b} \mu\left(\frac bd\right) a^d \qqand \delta_{1a}
    = \sum_{d|a} \mu(d).\] Suppose $n$ is odd, so $X_{(1,\frac n2)}$ is empty.
  We have (by Mobius inversion)
  \begin{align*}
    P_1 &=  \frac{1}{2n} \left( q^n - \#\{\hbox{points in smaller
                   extensions}\}\right)
    = \frac{1}{2n} \sum_{d|n} \mu\left(\frac nd\right) q^d.
  \end{align*}
  Since we only care about the parity of this number it suffices to consider the
  sum $\sum_{d|n} \mu(n/d) q^d$ modulo $4$.  Since $n$ is odd, $d$ must also
  always be odd; in particular, $q^d \equiv q\pmod 4$ for all $d$.  Thus, since
  $n > 1$,
  \[(2n)P_1 \equiv q \sum_{d|n} \mu\left(\frac nd \right) = 0 \pmod
    4.\]
  This completes the odd case.
  
  When $n$ is even, write $n = 2^m n'$; there are two types of orbits $(2,0)$
  and $(1,\frac n2)$.  Consider first orbits of type $(1,\frac n2)$.  A point in
  an orbit of type $(1,\frac n2)$ satisfies $-x = \Frob^{n/2} x$.  There are
  exactly $q^{n/2}-1$ of these; if any solution lies in a subextension of even
  index then it must lie in $\FF_q$, which contains exactly $2$ solutions.  Call
  a point $x$ \emph{good} if $x^{q^{n/2}-1} = 1$.  Let $L_1,\ldots,L_b$ be the
  maximal proper subfields of $\FF_q$ of odd index.  Then, using the principle
  of inclusion/exclusion (or Mobius inversion),
  \begin{align*}
    n P_2 &= (q^{n/2}-1) - \sum_{i=1}^b
                           \#\{\hbox{good points in }L_i\} + \sum_{i,j} \#\{\hbox{good points in }L_i
                           \cap L_j\} - \cdots\\
                         &= \sum_{\substack{d | n,\ \frac nd\ \mathrm{odd}}} \mu\left(\frac n d\right)
                           (q^{d/2}-1) = \sum_{d|n'} \mu\left(\frac{n'}{d}\right)(q^{2^{m-1}})^d - \delta_{1n'}. 
  \end{align*}
  Only the parity of $P_2$ matters, so it suffices to consider the right-hand
  side modulo $2^{m+1}$.  If $m > 1$, since $q$ is odd then
  $q^{2^{m-1}} \equiv 1 \pmod{2^{m+1}}$, so
  \[nP_2 = \sum_{d|n'} \mu\left(\frac{n'}{d}\right) - \delta_{1n'} \equiv 0
    \pmod{2^{m+1}}.\]
  On the other hand, if $m = 1$ we have $q^d \equiv q \pmod
  4$ and thus
  \[nP_2 \equiv (q - 1)\delta_{1n'} \pmod 4,\]
  thus giving the desired formula.
  
  Now consider $P_1$.  Using the fact that an orbit of type $(2,0)$ has $2n$
  points and an orbit of type $(1,\frac n2)$ has $n$ points, in terms of point
  counts over $\FF_{q^n}$
  \begin{align*}
    P_1 &= \frac{1}{2n} \left( q^n - \#\{\hbox{points in smaller
                   extensions}\} - \#\{\hbox{points in orbits of type }(1,\frac n2)\} \right)\\
                 &= \frac{1}{2n} \left(q^n - \#\{\hbox{points in smaller
                   extensions}\} - nP_2\right).
  \end{align*}
  By Mobius inversion,
  \[q^n - \#\{\hbox{points in smaller extensions}\} = \sum_{d|n} \mu\left(\frac
      nd\right) q^d.\] Thus the parity of $P_1$ can be determined by considering
  $2n\# P_1$ modulo $2^{m+2}$.  
  \[2n P_1 \equiv \bigg(\sum_{d|n} \mu\left(\frac nd \right) q^d\bigg) -
    \delta_{1n'}(q^{2^{m-1}}-1) \pmod{2^{m+2}}.\] Consider the first sum.  If
  $\ord_2 d < m-1$ then $\mu(n/d) = 0$.  Thus the only summands which are
  nonzero must have $\ord_2(d) = m$ or $m-1$.  Thus
  \[\sum_{d|n} \mu\left(\frac nd \right) q^d = \sum_{d|n'}
    \mu\left(\frac{n'}{d}\right) q^{2^md} - \sum_{d|n'}
    \mu\left(\frac{n'}{d}\right) q^{2^{m-1}d} = \sum_{d|n'}
    \mu\left(\frac{n'}{d}\right) (q^{2^md} - q^{2^{m-1}d}).\]
  Modulo $2^{m+2}$, $q^{2^m d} \equiv 1$.  As $d$ is odd, $q^{2^{m-1}d} \equiv
  q^{2^{m-1}}$  modulo $2^{m+2}$.  We can therefore conclude that
  \[2n\#X_{(2,0)} \equiv - 2\delta_{1n'}(q^{2^{m-1}}-1) \pmod{2^{m+2}}.\] In
  particular, if $n' > 1$ this is $0$.  If $m > 1$ then
  $q^{2^{m-1}}\equiv 1 \pmod{2^{m+1}}$, this must be $0$.  Lastly, if $n = 2$
  this is $-2 (q-1) \pmod 8$; in other words, $\# X_{(2,0)}$ is even if $q$ is
  $q \equiv 1 \pmod 4$ and odd if $q \equiv 3 \pmod 4$.
\end{proof}

\begin{remark}
  The question of whether $[\PP^1,1/x]$ is non-permutative (or even
  non-$0$-dimensional!) when $q \equiv 1 \pmod 4$ remains open.  
\end{remark}

The previous result is implies that multiplying by $-1$ is non-permutative if
$-1$ is not a square.   We can use the intuition behind this to show that
non-permutative elements always exist:
\begin{corollary} \label{cor:gen-root}
  Let $\FF_q$ be a finite field such that $\ord_2(q-1) = \ell$.  Let $\alpha$
  be a primitive $2^\ell$-th root of unity.  Then the element
  $[\PP^1, \alpha x]$ is non-permutative.
\end{corollary}

\begin{proof}
  When $\ell=1$ this is simply Corollary~\ref{cor:calc}, so we focus on the case
  $\ell > 1$.  By Proposition~\ref{prop:su-special}, keeping in mind that $m =
  2^\ell$, 
  \[\psi_{2^\ell}[\PP^1, \alpha x] = \bigg((-1)^{P_1+P_2}, 2^{\ell-1}
    \sum_{j=1}^\ell P_{2^j}\bigg).\]
  To show that $[\PP^1,\alpha x]$ is non-permutative it therefore suffices to
  check that $\sum_{j=1}^\ell P_{2^j}$ is odd.  In fact, we claim that
  $P_{2^\ell}$ is odd and $P_{2^j}$ is even for $1 \leq j < \ell$.

  First, consider $P_{2^\ell}$, which counts orbits of type $(1,1)$.  Solutions
  to $x^q = \alpha x$ correspond to solutions to
  \[a(q-1) \equiv \frac{q^{2^\ell}-1}{2^\ell}  \pmod{q^{2^\ell}-1}.\]
  Since $(q-1)/2^\ell \mathrel{|} q-1, \frac{q^{2^\ell}-1}{2^\ell},
  q^{2^\ell}-1$, solutions to this equation exist exactly when solutions to
  \[2^\ell a \equiv q^{2^\ell-1} + \cdots + q + 1 \pmod{2^\ell(q^{2^{\ell}-1} +
      \cdots + 1)}\]
  exist.  As there are $2^\ell$ terms in the right-hand side of the equivalence,
  solutions exist---and thus there are exactly $q-1$ different solutions to the
  original equation.  Note, in addition, that none of these are in extensions of
  lower degree.  If we instead consider solutions to $x^q = \alpha x$ in
  $\FF_{2^d}$ with $d < \ell$ then they correspond to solutions to
  \[2^\ell a \equiv q^{2^d-1} + \cdots + q + 1 \pmod{2^\ell(q^{2^d-1} + \cdots +
      q + 1}.\] The left-hand side of the equivalence is $0$ mod $2^\ell$, while
  the right-hand side is $2^d \not\equiv 0 \pmod{2^\ell}$; thus there are no
  solutions in lower extensions.  Thus
  \[P_{2^ell} = \frac{q-1}{2^\ell} \equiv 1 \pmod 2.\]

  Now consider $P_{2^{\ell-1}}$, which counts orbits of type $(2,2)$.  Solutions
  to $x^{q^2} = \alpha^2 x$ correspond to solutions to
  \[a(q^2 - 1) \equiv \frac{q^{2^\ell}-1}{2^{\ell-1}} \pmod{q^{2^\ell}-1}.\] As
  above, there are exactly $q^2-1$ solutions to this equation.  Now consider
  $x^{q^2} = \alpha^2 x$ over $\FF_{q^{2^d}}$ for $d < \ell$.  If $d = 0,1$
  there are no solutions, since $x^{q^2} = x$.  If $d \geq 2$ then
  solutions to this equation correspond to solutions to
  \[a(q^2-1) \equiv \frac{q^{2^d}-1}{2^{\ell-1}} \pmod{q^{2^d}-1}.\]
  Dividing all three terms by $\frac{q^2-1}{2^{\ell-1}}$ gives
  \[2^{\ell-1}a \equiv q^{2^d-2} + q^{2^d-4} + \cdots + q^2 + 1
    \pmod{2^{\ell-1}(q^{2^d-2} + q^{2^d-4} + \cdots + q^2 + 1)}.\]
  The left-hand side of the equivalence is $0$ mod $2^{\ell-1}$, but the
  right-hand side is equivalent to $2^{d-1} \not\equiv 0 \pmod{2^{\ell-1}}$,
  since $d < \ell$.  Thus there are no solutions, and thus no points over
  lower-degree extenstions.  However, every orbit of type $(1,1)$ also gives
  solutions to this equation, as do orbits of type $(1, 2^{\ell-1}+1)$ (which
  are in bijection with orbits of type $(1,1)$).  Thus
  \[P_{2^{\ell-1}} = \frac{1}{2^{\ell+1}}\left((q^2-1) - 2^{\ell+1}P_1\right) =
    \frac{1}{2^{\ell+1}}\left((q^2-1) - 2(q-1)\right) =
    \frac{(q-1)^2}{2^{\ell+1}}.\]
  Thus $\ord_2 P_{2^{\ell-1}} = 2\ell - (\ell+1) = \ell -1$; since $\ell > 1$
  this is positive, and thus $P_{2^{\ell-1}}$ is even.

  We now claim that for $1 \leq r \leq \ell$,
  \[P_{2^{\ell-r}} = \frac{(q^{2^{r-1}}-1)^2}{2^{\ell+r}},\]
  so that
  \[\ord_2 P_{2^{\ell-r}} = \ell + r - 2 \geq 1;\]
  in particular $P_{2^{\ell-r}}$ is always even, from which the result follows.
  We prove this by induction on $r$.  The base cases $r = 0,1$ were done above,
  so we proceed to the inductive step.  This is analogous to the base case
  $r=1$, although with slightly more bookkeeping. To compute $P_{2^{\ell-r}}$ we
  first count solutions to $x^{q^r} = \alpha^{2^r} x$.  In $\FF_{q^{2^\ell}}$
  solutions to this correspond to solutions to
  \[a(q^{2^r}-1) \equiv \frac{q^{2^\ell}-1}{2^{\ell-r}} \pmod{q^{2^\ell}-1}.\]
  Since all three terms are divisible by $q^{2^r}-1$ solutions to the equation
  exist, and thus there are $q^{2^r}-1$ solutions.  For $d < \ell$, solutions to
  the equation in $\FF_{q^{2^d}}$ correspond to solutions to
  \[a(q^{2^r}-1) \equiv \frac{q^{2^d}-1}{2^{\ell-r}} \pmod{q^{2^d}-1}.\]
  Dividing both sides by $\frac{q^{2^r}-1}{2^{\ell-r}}$ shows that solutions to
  this equation exist exactly when there exist solutions to
  \[2^{\ell-r}a \equiv q^{2^d-2^r} + \cdots + q^{2^d} + 1
    \pmod{2^{\ell-r}(q^{2^d-2^r} + \cdots + q^{2^d} + 1)}.\]
  Modulo $2^{\ell-r}$ the left-hand side is $0$ but the right-hand side is
  $2^{d-r}$, which is not $0$; thus there are no solutions over lower-degree
  extensions.  However, some of these extensions come from orbits of type
  $P_{2^{\ell-r'}}$ for $r' < r$ and we have
  \begin{align*}
    P_{2^{\ell-r}} &= \frac{1}{2^{\ell+r}}\bigg((q^{2^r}-1) - \sum_{j=0}^{r-1}
                     2^{r-j} \cdot 2^{\ell+j} P_{2^{\ell-j}}
                     \bigg) 
                   = \frac{1}{2^{\ell+r}}\left(q^{2^{r-1}}-1\right)^2,
  \end{align*}
  by applying the induction hypothesis and the two base cases inside the
  summation. 
\end{proof}

As an alternate approach, one can consider elliptic curves with complex
multiplication by $i$.

\begin{example}
  Let $k = \FF_q$ with $q\equiv 1 \pmod 4$.  Then the elliptic curve $E$ given
  by $y^2 = x^3 + x$ has an automorphism $\varphi: (x,y) \rgoesto (-x,iy)$.
  This has order 4 on all finite points except for those where $y = 0$, which
  are all over $\FF_q$.  Consider $\psi_2[E, \varphi]$.  There are two types of
  orbits: $(4,0)$ and $(2,1)$.  A point $(x,y)$ in an orbit of type $(2,1)$ has
  $(\bar x, \bar y) = (x, -y)$; in other words, if we write
  $\FF_{q^2} = \FF_q[\sqrt{\alpha}] $ for some $\alpha\in \FF_q$ then
  $x \in \FF_q$ and $y = y'\sqrt\alpha$.  Thus the point $(x,y')$ is an
  $\FF_q$-point of the curve $E'$ given by $\alpha y^2 = x^3 + x$, the quadratic
  twist of $E$.  Conversely, any finite point on $E'$ where $y \neq 0$
  corresponds to a point in an orbit of type $(2,1)$; thus the number of orbits
  of type $(2,1)$ is a quarter of the number of finite points with nonzero
  $y$-coordinate:
  \[\# X_{(2,1)} = \frac{1}{4}E'(\FF_q) - 1.\]

  Therefore
  \begin{align*}
    \# X_{(4,0)} &= \frac{1}{8}\left(E(\FF_{q^2}) - E(\FF_q) - 4 \#
                   X_{(2,1)}\right) 
                     = \frac18E(\FF_{q^2}) - \frac{q-1}{4}.
  \end{align*}

  Using Proposition~\ref{prop:stable-unstable} we conclude that
  \[\psi_2[E,\varphi] = \left((-1)^{\# E(\FF_{q^2})/8 - (q-1)/4}, \frac14 \#E'(\FF_q)-1\right) \in
    \Z/2\times \Z/2.\]
  In particular, $[E,\varphi]$ is non-permutative if $\# E'(\FF_q)$ is a
  multiple of $8$.

  For example, consider $q = 5$.  In this case $\# E(\FF_{q^2}) = 32$ and
  $\#E'(\FF_q) = 8$.  Thus this element is non-permutative.
\end{example}

As base change is also an $E_\infty$-map, these examples provide the tools to
determine that certain classes over infinite bases are also non-permutative. 

\begin{corollary}
  Let $k$ be a global or local field with a place of cardinality
  $q \equiv 3 \pmod{4}$.  Then $K_1(\Var_{\mathcal{O}_k})$ has non-permutative
  elements.  Alternately, if $k$ contains a place of cardinality
  $q \equiv 2^{\ell-1}+1 \pmod {2^\ell}$ and a primitive $2^{\ell-1}$-st root of
  unity then $K_1(\Var_{\mathcal{O}_k})$ has non-permutative elements.
\end{corollary}

\begin{proof}
  Consider the map $K(\Var_{\mathcal{O}_k}) \rto K(\Var_{\FF_q})$ induced by the
  reduction to $\FF_q$.  This is an $E_\infty$-map and thus the preimage of any
  non-permutative element is non-permutative.

  In the first case, consider the element $[\PP^1,-x]$; by
  Corollary~\ref{cor:calc} its image under base change is non-permutative, and
  therefore it is also non-permutative.

  In the second case, consider the automorphism of $\PP^1$ given by
  $x \rgoesto \alpha x$, for $\alpha$ a primitive $2^{\ell-1}$-st root of unity.
  The image of $[\PP^1,\alpha x]\in K_1(\Var_{\mathcal{O}_k})$ in
  $K_1(\Var_{\FF_q})$ is non-permutative by Corollary~\ref{cor:gen-root}, and is
  therefore non-permutative.
\end{proof}





\begin{remark}
  In \cite{CWZ-zeta}, the author and collaborators give an alternate
  construction of such a ``derived zeta function.''  By using the observation
  that $|X(L_n)|$ is exactly the number of fixed points of $\mathrm{Frob}_k^{n}$
  and using the Grothendieck--Lefschetz fixed point theorem, the collaborators
  construct a map of spectra $K(\Var_k) \rto K(\End(\Q_\ell))$ which on $K_0$ is
  exactly the zeta function (using a result of Almkvist \cite{grayson78} to show
  that the composition of that map with
  $K_0(\End(\Q_\ell)) \rto (1+t\Z\llbracket t\rrbracket,\times)$ is $Z(-,t)$).
  As $K_0(\End(\Q_\ell))$ is the rational Witt vectors, this produces a derived
  zeta function which recalls that zeta functions should be rational.  Our
  current construction cannot do that, as the data about all orbits of different
  sizes are independent.  However, the construction of \cite{CWZ-zeta} has the
  weakness that it was not a map of $E_\infty$-ring spectra, as it was
  constructed using two formal inverses to weak equivalences.  The construction
  was also significantly more complicated than the construction in this paper,
  leading to much more difficulty with analysis.  In future work, the author
  hopes to find a construction that unifies the strength of these two
  approaches, so that it can retain the rationality data as well as the ring
  structure.
\end{remark}

\section{Technical Preliminaries} \label{sec:technical}

In this section we review some of the technical preliminaries necessary for the
proofs.  Most of the results in this section can be found in \cite[Section
2]{Z-Kth-ass}; we revisit them here in the interest of readability.

\begin{definition}[{\cite[Definition 2.1]{zakharevich10}}]
  Let $\C$ be an assembler.  The category $\Tw(\C)$ is defined to have
  \begin{description}
  \item[objects] tuples $\SCob{A}{i}$, where $I$ is a finite set and each
    $A_i$ is a noninitial object in $\C$.
  \item[morphisms] A morphism $f:\SCob{A}{i} \rto \SCob{B}{j}$ is a map of
    finite sets $f:I \rto J$ (called the \emph{set map}), together with
    morphisms $f_i:A_i \rto B_{f(i)}$ in $\C$, for each $i\in I$ (called the
    \emph{component maps}).  The component maps must satisfy the condition
    that for $i\neq i'$, if $f(i) = f(i')$ then $f_i$ and $f_{i'}$ are disjoint.
  \item[composition] The composition of $f:\SCob{A}{i} \rto \SCob{B}{j}$ and
    $g:\SCob{B}{j} \rto \SCob{C}{k}$ is given by the set map $g\circ f$,
    together with the component maps $g_{f(i)} \circ f_i:A_i \rto C_{gf(i)}$.
  \end{description}

  The category $\W(\C)$ is the subcategory of $\Tw(\C)$ containing all morphisms
  $\SCob{A}{i} \rto \SCob Bj$ such that for all $j\in J$, the family $\{f_i;A_i
  \rto B_j\}_{j\in f^{-1}(i)}$ is a finite disjoint covering family.
\end{definition}

In this paper, we use two distinct constructions of the $K$-theory of an
assembler: the $\Gamma$-space definition from \cite{Z-Kth-ass} and the
Waldhausen category definition from \cite{Z-ass-pi1}.  In \cite[Theorem
2.1]{Z-ass-pi1} it is shown that these two constructions produce equivalent
$K$-theories for closed assemblers; in this paper, we will also show that this
equivalence respects the monoidal structure.  We give a short review of these
definitions here.  Write $\Gamma\Spc$ for the category of $\Gamma$-spaces,
$\WaldCat$ for the category of Waldhausen categories, and $\Sp$ for the category
of symmetric spectra.

We begin by recalling the two definitions.

\begin{definition}[{\cite[Definition 2.12]{Z-Kth-ass}}] \label{def:Kass}
  Let $X$ be a pointed set; write $X^\circ \defeq X \bs \{*\}$.  For an
  assembler $\C$, we write $X \sma \C$ for the assembler
  $\bigvee_{x\in X^\circ} \C$; here, the wedge product of assemblers is obtained
  by taking their unions and associating the initial objects.  This gives a
  functor $\FinSet_* \times \Assemb \rto \Assemb$.  In addition, there is a
  natural transformation $\cdot \sma N\W(\C) \rto N\W(\cdot \sma \C)$ given
  for each $X$ by the composition
  \[X \sma N\W(\C) \cong \bigvee_{X^\circ} N\W(\C) \rto
    N\left(\bigoplus_{X^\circ} \W(\C)\right) \cong N\W(X \sma \C).\]  
  
  For an assembler $\C$, we define
  \[K^\Gamma(\C) \defeq \mathbf{B}(X \rgoesto \W(X \sma \C)).\]
  Here, $\mathbf{B}$ is the classifying spectrum functor which takes a
  $\Gamma$-space to a spectrum.  When not comparing this construction to $K^W$
  (defined below), we write $K$ instead of $K^\Gamma$.
\end{definition}

\begin{definition}[{\cite[Definition 1.7]{Z-ass-pi1}}]
  Suppose that $\C$ is a closed assembler.  We define the category $\SC(\C)$
  to have
  \begin{description}
  \item[objects] $\ob \Tw(\C)$
  \item[morphisms] A morphism $f:\SCob Ai \rto \SCob Bj$ is represented by a span
    \[\SCob Ai \lto^p \SCob Ck \rto^\sigma \SCob Bj.\]
    Here $p$ is a morphism in $\Tw(\C)$, and $\sigma$ is represented by a set
    map $\sigma: K \rto J$, together with component maps $\sigma_k:C_k \rto
    B_{\sigma(k)}$ which are isomorphisms in $\C$.
  \item[composition] The composition of two morphisms $f: \SCob{A}{i} \rto
    \SCob{B}j$ and $g: \SCob Bj \rto \SCob Ck$ represented by a diagram
    \begin{squisheddiagram}
      { & \SCob{A'}{i'} & & \SCob{B'}{j'} \\
        \SCob Ai && \SCob Bj && \SCob Ck \\};
      \to{1-2}{2-1}_p \to{1-2}{2-3}^\sigma  \to{1-4}{2-3}_q \to{1-4}{2-5}^\tau
    \end{squisheddiagram}
    is defined by pulling back $q$ along $\sigma$ and composing down the two
    sides.  It is necessary to check that such a pullback produces a
    well-defined composition; see\cite[Lemma 6.4]{zakharevich10}.
  \end{description}
  We give $\SC(\C)$ the structure of a Waldhausen category by defining
  \begin{description}
  \item[cofibrations] to be those morphisms where $p$ is in $\W(\C)$ and
    $\sigma$ has an injective set map, and
  \item[weak equivalences] to be those cofibrations where $\sigma$ has a
    bijective set map.
  \end{description}
  For a closed assembler $\C$, we define
  \[K^W(\C) \defeq K(\SC(\C)).\]
\end{definition}

In assemblers it is often possible to compute the quotient of the $K$-theory of
an assembler by the $K$-theory of a subassembler by simply ``removing'' the
objects of the subassembler.

\begin{definition}[{\cite[Definition 2.9]{Z-Kth-ass}}]
  Let $\C$ be an assembler and $\D$ a sieve in $\C$.  The assembler $\C \bs \D$
  is defined to have as its underlying category the full subcategory of $\C$
  containing all objects not in $\D^\circ$.  A family
  $\{f_i:A_i \rto A\}_{i\in I}$ is a covering family in $\C\bs\D$ if there
  exists a family $\{f_j:A_j \rto A\}_{j\in J}$ with each $A_j\in \D$ such that
  $\{f_i:A_i \rto A\}_{i\in I\cup J}$ is a covering family in $\C$.
\end{definition}

Often it is the case that $K(\C)/K(\D) \simeq K(\C\bs\D)$; see \cite[Theorem
D]{Z-Kth-ass} for more detail.  Here, we consider a situation where this
does \emph{not} hold, as it will be important intuition for the construction of
the monoidal structure on $c\Assemb$.

Consider an object $A$ which has an empty covering family.  The morphism
$\{\}_\emptyset \rto \{A\}_{\{*\}}$ is a weak equivalence in $\SC(\C)$.  Thus,
morally speaking, $A$ should not contribute to the $K$-theory of $\C$.  However,
this can be deceiving, as the underlying categorical structure of $\C$ can
contribute to the $K$-theory of $\C$ despite this.  To help illustrate this, we
present an example, which will be helpful in understanding the difference
between $\C\boxtimes \D$ and $\C \sma \D$ in Section~\ref{sec:monoidal}:

\begin{example} \label{ex:square}
  Let $\C$ be the assembler with the following underlying category:
  
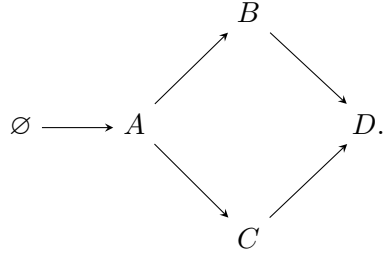
\begin{diagram}
    { & & B \\
      \initial & A & & D. \\
      & & C \\};
    \to{2-1}{2-2} \to{2-2}{1-3} \to{1-3}{2-4}
    \to{2-2}{3-3} \to{3-3}{2-4}
  \end{diagram}
  The topology on $\C$ is generated by the covering families $\{B \rto D, C \rto
  D\}$ and the empty covering families of $A$ and $\initial$.  Note that $D$ has
  \emph{no} finite disjoint covering families.  Thus, despite the fact that
  $D$ is covered by $B$ and $C$,
  \[K(\C) \simeq \S\vee\S\vee\S,\]
  with one copy of $\S$ for each of $B$, $C$, and $D$.

  Now consider $\C\bs \{\initial \rto A\}$; this assembler is given by the
  diagram
  
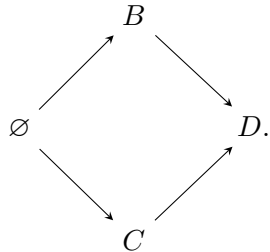
\begin{diagram}
    { & B \\
      \initial & & D. \\
      & C\\};
    \to{2-1}{1-2} \to{1-2}{2-3} \to{2-1}{3-2} \to{3-2}{2-3}
  \end{diagram}
  The topology on this assembler is generated by the covering family $\{B \rto
  D, C \rto D\}$ and the empty covering family on $\initial$.  Here, $D$
  \emph{does} have a finite disjoint covering family; thus
  \[K(\C \bs \{\initial \rto A\}) \simeq \S \vee \S,\]
  with one copy of $\S$ for each of $B$ and $C$.
\end{example}

\section{A monoidal structure on the category of assemblers}
\label{sec:monoidal}

The goal of this section is to construct a symmetric monoidal structure on
$c\Assemb$ in such a way that the $K$-theory functor is symmetric monoidal and
$\S^\Assemb$ is the unit.

We begin with a helper definition.

\begin{definition}
  Let $\C,\D$ be two closed assemblers. The assembler $\C\boxtimes \D$ has as
  its underlying category the category $\C\times \D$, and its topology is
  generated by the coverage in which the covering families are families
  $\{(A_i,B_j) \rto (A,B)\}_{(i,j)\in I\times J}$, where
  $\{A_i \rto A\}_{i\in I}$ is a covering family in $\C$ and
  $\{B_j \rto B\}_{j\in J}$ is a covering family in $\D$.
\end{definition}

\begin{lemma}
  $\C\boxtimes \D$ is a closed assembler. 
\end{lemma}

\begin{proof}
  First we need to check that the Grothendieck topology is well-defined.  To
  check this we simply need to check that the pullback of a family in the
  coverage is still in the coverage; this follows because covering families in
  $\C$ and $\D$ are closed under pullbacks.

  Axioms (I) and (M) follow directly from the definition.  Axiom (R) holds
  because $\C\boxtimes \D$ has pullbacks.
\end{proof}

\begin{remark}
  This construction is analogous to the construction of the product topology.
  In the product topology, the generating open sets are the products of the opens
  in the two categories.  More concretely, suppose $U \subseteq X$ is covered by
  $\{A_1,\ldots,A_n\}$ and $V\subseteq Y$ is covered by $\{B_1,\ldots,B_m\}$.
  Then to cover $U\times V$ we need to take
  $\{A_i \times B_j \,|\, 1 \leq i \leq n, 1 \leq j \leq m\}$.

  Contrast this with the usual disjoint union topology on $\C\times \D$ (where
  $\C$ and $\D$ are sites) where $\{(A_i,B_i) \rto (A,B)\}_{i\in I}$ is a
  covering family if $\{A_i \rto A\}_{i\in i}$ and $\{B_i \rto B\}_{i\in I}$ are
  both covering families.  
\end{remark}

Let
$\alpha_{\C,\D,\E}: (\C\boxtimes \D) \boxtimes \E \rto \C\boxtimes (\D\boxtimes
\E)$ be the functor taking $((A,B),C)$ to $(A,(B,C))$.  Let
$\gamma_{\C,\D}: \C\boxtimes \D \rto \D\boxtimes \C$ be the functor taking
$(A,B)$ to $(B,A)$.  Let $\lambda_\C: \S^\Assemb\boxtimes \C \rto \C$ be the
projection onto the second coordinate and let
$\rho_\C: \C \boxtimes \S^\Assemb \rto \C$ be the projection onto the first
coordinate.

\begin{lemma} \label{lem:boxtimes_nats}
  The natural transformations $\alpha$, $\gamma$, $\lambda$ and $\rho$ satisfy
  all of the axioms of a symmetric monoidal structure except the condition that
  $\lambda$ and $\rho$ be natural isomorphisms. 
\end{lemma}

\begin{proof}
  The only part that is not direct from the definitions is checking that
  $\lambda_\C$ and $\rho_\C$ are well-defined morphisms of assemblers.  We focus
  on $\lambda_\C$; the result for $\rho_\C$ will follow analogously.  Since the
  topology on $\S^\Assemb\boxtimes\C$ is generated by a pretopology and since
  $\lambda_\C$ commutes with pullbacks (since pullbacks in $\S^\Assemb\boxtimes \C$ are
  done coordinatewise) it suffices to check that for any covering family $\F$,
  $\lambda_\C\F$ is a covering family.  A covering family in $\S^\Assemb\boxtimes \C$ is
  a finite refinement of families in the coverage.  The projection of a family
  in the coverage is a covering family.  The refinement of a covering family
  can either refine the projection, or it can add some morphisms to the covering
  family (which still keeps it a covering family).  Either way, the projection
  of a covering family is a covering family, as desired.
\end{proof}

Thus $\boxtimes$ does not produce a symmetric monoidal structure on $\Assemb$.
To make this into a monoidal structure it is necessary to rectify this problem.

\begin{definition}
  We consider $\C\mathrel{\tilde\vee}\D$ to be the full subassembler of
  $\C\boxtimes \D$ containing those objects where one coordinate or the other is
  the initial object.  The subassembler $\C\mathrel{\tilde\vee}\D$ is a sieve in
  $\C\boxtimes \D$.

  Define
  \[\C\sma \D \defeq (\C\boxtimes \D) \bs (\C\mathrel{\tilde\vee}\D).\]
\end{definition}

The relationship of $\C\sma\D$ to $\C\boxtimes \D$ has the exact flavor of
Example~\ref{ex:nonsaturated}, with objects in $\C\mathrel{\tilde \vee}\D$ being
obstructions to objects in $\C\boxtimes\D$ being disjoint.  Once these objects
are removed, moreover, the natural transformations
$\alpha, \gamma, \lambda,\rho$ are all well-defined with $\boxtimes$ replaced by
$\sma$, since \emph{as categories} $\C\sma\D$ can be thought of as a full
subcategory of $\C\boxtimes \D$.

\begin{lemma} \label{lem:smasym}
  $(c\Assemb, \S^\Assemb, \sma)$ is a symmetric monoidal category.
\end{lemma}

\begin{proof}
  By Lemma~\ref{lem:boxtimes_nats} all that remains to show is that $\lambda$
  and $\rho$ are natural isomorphisms.

  We check that $\lambda_\C$ is an isomorphism $\S^\Assemb\sma \C \rto \C$.  The
  objects of $\S^\Assemb\sma \C$ are the initial object and pairs $(*,A)$ with $A\in
  \C$.  $\lambda_\C$ takes the first to the initial object and the second to
  itself.  The structure on morphisms is analogous.  Since the functor is a
  bijection on both objects and morphisms, it is an isomorphism, as desired.

  We must also check that $\lambda_{\S^\Assemb} = \rho_{\S^\Assemb}$.  This is the case because $\S^\Assemb
  \sma \S^\Assemb$ has a unique nontrivial map to $\S^\Assemb$; since both $\lambda_{\S^\Assemb}$ and
  $\rho_{\S^\Assemb}$ areisomorphisms, they must be equal.
\end{proof}

Each object of $\C\tilde\vee\D$ sitting inside $\C\boxtimes \D$ has an empty
covering family.  It is thus tempting to conclude that $K(\C\boxtimes \D)$ may
already have the correct symmetric monoidal structure, at least up to homotopy.
In general this is \emph{not} the case, as structures similar to those in
Example~\ref{ex:square} arise.  In fact, most objects in
$\C^\circ\times \D^\circ$ have \emph{no} nontrivial finite disjoint covering
families.  Indeed, suppose that $(A,B)$ is an object for which the empty family
is not a covering family.  Then any nontrivial covering family will have two
elements that share a coordinate, and will therefore not be disjoint.  On the
other hand, in $\C\sma \D$, any pair of finite disjoint covering families
produces a finite disjoint covering family, since only one of the coordinates
being disjoint is sufficient for disjointness.

For a more concrete example, consider $K(\S^\Assemb \boxtimes \S^{\Assemb})$.
$\S^\Assemb \boxtimes \S^\Assemb$ has three noninitial objects and no nontrivial
finite disjoint covering families.  Thus $K(\S^\Assemb\boxtimes \S^\Assemb)
\simeq \S \vee\S\vee\S$.  If the functor were correctly monoidal it would
instead be $\S$, so we see that $\boxtimes$ is not the desired structure, even
up to homotopy.

The monoidal structure of assemblers gives rise to an interesting phenomenon: in
general, the category $\SC(\C\sma\D)$ will \emph{not} be saturated.  We give an
explicit example to illustrate how this can arise, and we stress that such
examples are the norm and not the exception:

\begin{example} \label{ex:nonsaturated}
  Let $\C$ be the assembler $\Seg$, discussed in Example~\ref{ex:seg}. In
  $\SC(\Seg\sma\Seg)$ there is the following diagram of morphisms:
  
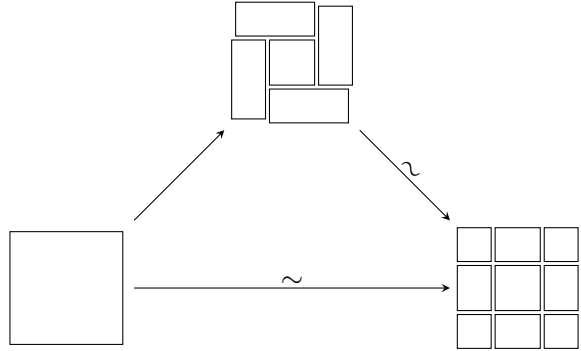
\begin{center}
  \begin{tikzpicture}[scale=1.5]
    \draw (0,0) rectangle (1,1);

    \draw[xshift=-1pt] (2,2) rectangle (2.3,2.7);
    \draw[yshift=-1pt] (2.3,2) rectangle (3,2.3);
    \draw[yshift=1pt] (2,2.7) rectangle (2.7,3);
    \draw[xshift=1pt] (2.7,2.3) rectangle (3,3);
    \draw (2.3,2.3) rectangle (2.7,2.7);
    
    \draw[xshift=-1pt,yshift=-1pt] (4,0) rectangle (4.3,0.3);
    \draw[xshift=-1pt] (4,0.3) rectangle (4.3,0.7);
    \draw[xshift=-1pt, yshift=1pt] (4,0.7) rectangle (4.3,1);
    \draw[yshift=-1pt] (4.3,0) rectangle (4.7,0.3);
    \draw (4.3,0.3) rectangle (4.7,0.7);
    \draw[yshift=1pt] (4.3,0.7) rectangle (4.7,1);
    \draw[xshift=1pt, yshift=-1pt] (4.7,0) rectangle (5,0.3);
    \draw[xshift=1pt] (4.7,0.3) rectangle (5,0.7);
    \draw[xshift=1pt, yshift=1pt] (4.7,0.7) rectangle (5,1);

    \draw[->] (1.1,1.1) to  (1.9,1.9);
    \draw[->] (3.1,1.9) to node[above=-.7ex,sloped] {$\sim$} (3.9,1.1);
    \draw[->] (1.1,0.5) to node[above=-.7ex,sloped] {$\sim$} (3.9,0.5);
  \end{tikzpicture}
\end{center}
A weak equivalence is a morphism that can be written as a finite composition of
decompositions into ``grids'' on each rectangle; the two marked morphisms are
therefore weak equivalences, but the unmarked one is not.  Thus
$\SC(\Seg\sma\Seg)$ does not satisfy the saturation axioms.

\end{example}

\section{The interaction of $K$-theory and the monoidal
  structure} \label{sec:Kmonoidal}

In an ideal world, the $K$-theory functor would be monoidal and we could
construct ring spectra simply by finding monoid objects inside $\Assemb$.
However, that is not the case: even the category of pointed finite sets does not
produce an honest ring spectrum, but rather an $E_\infty$-spectrum, as it is not
possible to make a completely rigid model of both of the monodial structures
(disjoint union and product) on finite sets.

However, we can produce the next best thing: a bipermutative category.
\begin{definition}
  A category $\C$ is \emph{permutative} if it is equipped with a functor
  $\oplus: \C\times\C \rto \C$, an object $0\in \C$, and a natural isomorphism
  $\gamma: a\oplus b \cong b \oplus a$ satisfying the extra
  conditions that
  \begin{itemize}
  \item[(1)] $a \oplus (b \oplus c) = (a\oplus b) \oplus c$,
  \item[(2)] $a \oplus 0 = a = 0 \oplus a$, and
  \item[(3)] $\gamma_{a,0)} = 1_a$ and the following diagrams commute:
    \[
      \begin{inline-diagram}
        { a \oplus b & b\oplus a & a \oplus b\\};
        \to{1-1}{1-2}^\gamma \to{1-2}{1-3}^\gamma
        \node (m-100-100) at (m-1-1) {\phantom{$a\oplus b$}};
        \diagArrow{bend right}{100-100}{1-3}_{1_{a\oplus b}}
      \end{inline-diagram}
      \qquad
      \begin{inline-diagram}
        { a\oplus b \oplus c & & c\oplus a \oplus b \\
          & a\oplus c \oplus b. \\};
        \to{1-1}{1-3}^\gamma
        \to{1-1}{2-2}_{1\oplus \gamma} \to{2-2}{1-3}_{\gamma\oplus 1}
      \end{inline-diagram}
    \]
  \end{itemize}
  In other words, a permutative category is a symmetric monoidal category with
  strict associativity and unit.  This is referred to as the \emph{additive}
  structure on the permutative category.

  A permutative category $\C$ is \emph{bipermutative} if it is equipped with a second
  permutative structure $(\C,\otimes,1)$ (which is referred to as the
  \emph{multiplicative} structure)and natural distributivity maps
  \[d_l: (a\otimes b) \oplus (a'\otimes b) \rto (a\oplus a') \otimes b\]
  and
  \[d_r: (a\otimes b) \oplus (a \otimes b') \rto a \otimes (b\oplus b')\]
  satisfying certain compatibility requirements, described in \cite[Definition
  3.3, 3.6]{elmendorfmandell}.
\end{definition}

It is not immediately obvious why distributivity maps cannot be ``rigidified''
away, when it is known why monoidal structures can generally be replaced with
rigid versions.  To help with this, and as it will be used in
Proposition~\ref{prop:biperm}, we give an explicit description of a
bipermutative structure on the category of finite sets.

\begin{example} \label{ex:finset}
  Let $\FinSet$ be the category with
  \begin{description}
  \item[objects] the sets $\emptyset$ and $\{1,\ldots,n\}$ for all natural
    numbers $n$ and
  \item[morphisms] functions between finite sets.
  \end{description}

  The additive permutative structure on finite sets is given by disjoint union,
  where we think of ``concatenating the two sets in order''.  Thus in
  $\{1,\ldots,k\} \oplus \{1,\ldots,\ell\} = \{1,\ldots,k+\ell\}$ we think of
  the first $k$ elements as coming from $\{1,\ldots,k\}$ and the rest as coming
  from $\{1,\ldots,\ell\}$, in the correct order.  (Although morphisms are not
  required to preserve order, we keep track of it here so as to analyse the
  symmetry more precisely.)  The map $\gamma$ is the $k,\ell$-shuffle which
  preserves the order of the two sets and moves the elements past one another.

  The multiplicative permutative structure is via the cartesian product of sets,
  in which the isomorphism $\{1,\ldots,k\} \times \{1,\ldots,\ell\} \cong
  \{1,\ldots,k\ell\}$ is given via the lexicographic ordering of pairs.  With
  this second structure, the natural transformation $d_r$ is the identity map,
  but the natural transformation $d_l$ is \emph{not}, as illustrated below:
  
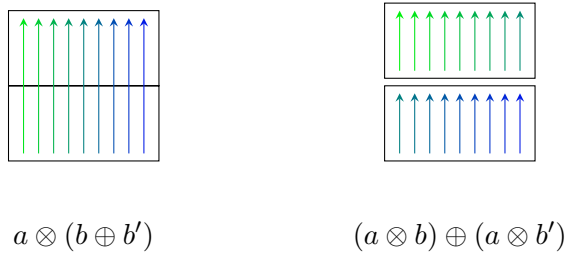
\begin{center}
    \begin{tikzpicture}
      \node at (1,-1) {$a\otimes (b\oplus b')$};
      \node at (6,-1) {$(a \otimes b)\oplus (a \otimes b')$};
      \draw (0,0) rectangle (2,1) (0,1) rectangle (2,2);
      \foreach \i in {10,20,30,40,50,60,70,80,90} 
          \draw[blue!\i!green,->] ({\i/50},0.1) -- ({\i/50},1.9);

      \draw (5,0) rectangle (7,1);
      \draw (5,1.1) rectangle (7,2.1);
      \foreach \i in {5,10,15,20,25,30,35,40,45} 
        \draw[blue!\i!green,->] ({\i/25+5},1.2) -- ({\i/25+5},2);
      \foreach \i in {50,55,60,65,70,75,80,85,90}
        \draw[blue!\i!green,->] ({(\i-45)/25+5},0.1) -- ({(\i-45)/25+5},0.9);
    \end{tikzpicture}
  \end{center}
  In this picture, the rectangle represents the set of pairs, with the arrows
  showing the induced ordering, with greener arrows before bluer arrows.  Note
  that in the two pictures the orderings are not the same.
\end{example}

We now investigate the structures on assemblers that produce morphisms of
bipermutative categories.  We begin with a simple observation.

\begin{lemma}
  Let $\C$ be an assembler.  Then $\W(\C)$ is a permutative category with the
  permutative structure induced by the permutative structure on $\FinSet$. 
\end{lemma}

\begin{proof}
  The category $\W(\C)$ has as objects tuples $\{A_i\}_{i\in I}$ with $I\in
  \FinSet$ and $A_i\in \C^\circ$ for all $i$.  We define
  \[\{A_i\}_{i=1}^n \oplus \{B_j\}_{j=1}^m = \{C_k\}_{k=1}^{n+m},\]
  where $C_k = A_k$ if $k \leq n$ and $C_k = B_{k-n}$ for $k > n$. The $0$
  object is the empty tuple $\{\}_\emptyset$.  As $\FinSet$ has a permutative
  structure, this inherits the same structure.
\end{proof}

Given a sufficiently strict monoidal product on $\C$, this permutative structure
can be extended to a bipermutative structure.

\begin{notation}
  Given $I,J\in \FinSet$, write $I\otimes J$ for the set $I\times J$ with the
  lexicographic ordering, associated via this ordering with the set
  $\{1,\ldots,|I|\cdot|J|\}$.  We will write $(i,j)\in I\otimes J$ for the
  image of $(i,j)$ under this ordering.
\end{notation}

\begin{definition}
  Let $\C$ be an assembler with a multiplication $\mu: \C \sma \C \rto \C$.  A
  \emph{symmetry for $\mu$} is a natural isomorphism
  $\gamma: \mu \Rto \mu\circ \tau$, where $\tau$ swaps the two factors of $\C$,
  satisfying the extra condition that $\gamma_{A,B} \circ \gamma_{B,A} =
  1_{\mu(A,B)}$ for all $A,B$.  
\end{definition}

\begin{proposition} \label{prop:biperm} Let $\C$ be a monoid object in $\Assemb$
  with product $\mu: \C\sma \C \rto \C$.  Then the permutative structure on
  $\W(\C)$ extends to a ring structure, with
  \[\{A_i\}_{i\in I} \otimes \{B_j\}_{j\in J} \cong \{\mu(A_i,B_j)\}_{(i,j)\in
      I\otimes J}.\]  If in addition $\C$ is equipped with a symmetry for $\mu$
  then this ring structure is a bipermutative structure.
\end{proposition}

\begin{proof}
  First we must check that the given tensor product is a well-defined functor 
  $\otimes:\W(\C) \times \W(\C) \rto \W(\C)$.   We factor this as
  \[\W(\C) \times \W(\C) \rto^\nu \W(\C\sma\C) \rto^{\W(\mu)} \W(\C).\]
  Here, we define $\nu$ by
  \[\nu(\SCob{A}{i}, \SCob{A'}{i'})  = \{(A_i,A'_{i'})\}_{(i,i')\in I\times
      I'}\] on objects, and define it on morphisms by taking the pair
  $f: \SCob Ai \rto \SCob Bj$ and $f': \SCob{A'}{i'} \rto \SCob{B'}{j'}$ to the
  morphism defined by the set map $f\times f': I\times I' \rto J\times J'$ and
  the component maps
  $(f_i, f_{i'}'): (A_i, A'_{i'}) \rto (B_{f(i)}, B'_{f'(i')})$.  It remains to
  check that this is well-defined: that if both $f$ and $f'$ were morphisms in
  $\W(\C)$ then $\nu(f,f')$ is a morphism in $\W(\C\sma \C)$.  In particular it is
  necessary to check:
  \begin{description}
  \item[disjointness] Given distinct $(i_0,i'_0),(i_1,i'_1)\in I\times I'$ with $f(i_0) = f(i_1)$ and
    $f'(i_0') = f'(i_1')$ we must check that $(f_{i_0},f'_{i'_0})$ and
    $(f_{i_1},f'_{i'_1})$ are disjoint.  
    it suffices to check that this is equal to $\initial$.  However, since the
    pairs were distinct, one of the two coordinates must be different; suppose
    WLOG that it is the first one, so that $i_0 \neq i_1$.  Then the maps
    $f_{i_0}$ and $f_{i_1}$ are disjoint, so $A_{i_0}\times_{B_{f(i_0)}} A_{i_1}
    = \initial$.  But then the pullback has a single initial coordinate when
    computed inside $\C\times \C$, and is therefore equal to $\initial$ in
    $\C\sma \C$, as desired.
  \item[covering] A morphism $f:\SCob Ai \rto \SCob Bj$ in $\W(\C)$ is a
    collection of finite disjoint covering families $\{f_i;A_i \rto B_j\}_{i\in
      f^{-1}(j)}$.  We have already checked that $\nu(f,f')$ is a collection of
    disjoint morphisms; it remains to check that they give covering families.
    In particular, it must be the case that for all $(j,j')\in J\times J'$, the
    family
    \[\{(f_i,f'_{i'}): (A_i,A'_{i'}) \rto (B_j, B'_{j'})\}_{(i,i')\in
        (f,f')^{-1}(j,j')}\] is a covering family.  It is an element of the
    coverage that generates the topology on $\C \sma \C$, so it is a covering
    family, as desired.
  \end{description}

  The unit map for $\C$ is a morphism of assemblers $\S \rto \C$; denote by
  $\nu\in \C$ the image of $*$.  We claim that $\otimes$ is strictly associative
  and $\{\nu\}_{\{1\}}$ is a strict unit for $\otimes$.

  We have that
  \[(\SCob{A}{i} \otimes \SCob{B}{j}) \otimes \SCob Ck =
    \{\mu(\mu(A_i,B_j),C_k)\}_{((i,j),k)\in (I\otimes J)\otimes K}\]
  and
  \[\SCob{A}{i} \otimes (\SCob{B}{j} \otimes \SCob Ck) =
    \{\mu(A_i,\mu(B_j,C_k))\}_{((i,j),k)\in I\otimes (J\otimes K)}.\]
  Since $\FinSet$ is bipermutative, the two indexing sets are equal.  Since $\C$
  is a monoid object, $\mu$ is strictly associative, so the two objects are
  equal, as well.  Thus $\otimes$ is strictly associative.

  Similarly,
  \[\{\nu\}_{\{1\}} \otimes \SCob Ai = \{\mu(\nu, A_i)\}_{(1,i)\in \{1\}\otimes
      I};\]
  since $\mu$ is strictly unital, this is equal to $\SCob Ai$, as desired.  The
  analogous proof works for $\SCob Ai \otimes \{\nu\}_{\{1\}}$.

  If we have the structure map $\gamma$ then symmetry also holds.  Indeed, the
  map
  \[\SCob Ai \otimes \SCob Bj = \{\mu(A_i,B_j)\}_{(i,j)\in I\otimes J} \rto
    \{\mu(B_j,A_i)\}_{(j,i)\in J\otimes I} = \SCob Bj \otimes \SCob Ai\]
  is given by sending $(i,j)\in I\otimes J$ to $(j,i)\in J\otimes I$ using the
  symmetry map in $\FinSet$, and mapping $\mu(A_i,B_j) \rto \mu(B_j,A_i)$ via
  $\gamma_{A_i,B_j}$.  This satisfies the relation for the symmetry map in a
  permutative category because the set map does (as $\FinSet$ is a
  bipermutative category) and the components maps satisfy it by the condition on
  $\gamma$. 

  We now turn to checking the relations for the ring structure.  We define
  distributivity maps
  \begin{align*}
    &d_l: (\SCob{A}{i} \otimes \SCob Bj)\oplus (\SCob{A'}{i'}\otimes \SCob Bj)
      \rto (\SCob Ai \oplus \SCob{A'}{i'}) \otimes \SCob Bj \\
    & d_r: (\SCob Ai \otimes \SCob Bj) \oplus (\SCob Ai \otimes \SCob{B'}{j'})
      \rto \SCob Ai \otimes (\SCob Bj \oplus \SCob{B'}{j'})
  \end{align*}
  to be induced from the distributivity maps on $\FinSet$, with identity maps as
  the components.  

  We follow the naming from \cite[Definition 3.3, Definition
  3.6]{elmendorfmandell}.  Axiom (a) holds because the product of any set with
  the empty set is empty.  Axioms (b), (c), (d), (e) and (f) hold because they hold
  in $\FinSet$ and all component maps in the given diagrams are identities.

  Given the extra structure of the symmetry $\gamma$, Axiom (e') holds in
  $\FinSet$, and over each index this reduces to the diagram
  \begin{diagram}
    { \mu(A,B) & \mu(A,B) \\
      \mu (B,A) & \mu(B,A) \\};
    \arrowsquare{1}{\gamma_{A,B}}{\gamma_{A,B}}{1}
  \end{diagram}
  which commutes.
\end{proof}

\begin{corollary} 
  The $K$-theory of any monoid object in $\C$ is an $A_\infty$-ring spectrum.
  The $K$-theory of a monoid object equipped with a symmetry for the
  multiplication is an $E_\infty$-ring spectrum.
\end{corollary}

Using this we can construct some examples of $E_\infty$-ring spectra.

\begin{example}
  For a group $G$, let $\Fin_G$ be the assembler with
  \begin{description}
  \item[objects] pairs of an integer $m \geq 0$ together with a finite set of
    tuples of integers of length $m$ equipped with a $G$-action, and
  \item[morphisms] $G$-equivariant inclusions of sets.
  \end{description}
  This has a multiplication $\mu: \Fin_G \sma \Fin_G \rto \Fin_G$ taking a pair
  $(m,S) \sma (n,T)$ to $(m+n, S\times T)$, where an element in $S\times T$ is
  modeled as a tuple of length $m+n$, except when $m = 1$ and $|S| = 1$, in
  which case $(m, S) \sma (n,T) = (n,T)$ (and similarly for the case when $n=1$
  and $|T| = 1$).  The $G$-action on $S\times T$ is diagonal.

  The multiplication is strictly associative and strictly unital, with unit
  $(1, \{1\})$.  It is also equipped with a symmetry for the multiplication,
  taking $S\times T$ to $T\times S$ via a shuffle swapping the first $m$ and
  last $n$ coordinates applied to each tuple.

  Thus, by the corollary, $K(\Fin_G)$ is an $E_\infty$-ring spectrum.
\end{example}

Moreover, \cite[Theorem 9.3.8]{elmendorfmandell} immediately implies the
following:
\begin{proposition}
  A symmetric monoid map of symmetric monoid objects in assemblers induces an
  $E_\infty$-map on the $K$-theories.
\end{proposition}

Unfortunately, much of the time we do not immediately have a strict model for
the multiplication.  Consider the example we are most interested in: varieties.
Given varieties $X$ and $Y$ over $k$ the product should be the fiber product
$X\times_k Y$.  However, to make $\Var_k$ a monoid object, we must be able to
model this fiber product \emph{rigidly}.  Although this is possible we prefer to
apply a more general technique and work with strong morphisms of monoidal
assemblers.

\begin{definition} \label{def:symmonass}
  A \emph{(symmetric) monoidal assembler} is an assembler $\C$ equipped with a
  map $\mu: \C \sma \C \rto \C$ and a map $\eta: \S \rto \C$ satisfying the
  axioms of a (symmetric) monoidal category.  More concretely, a (symmetric)
  monoidal assembler is equipped with a natural isomorphism $\alpha_{A,B,C}:
  \mu(A,\mu(B,C)) \rto \mu(\mu(A,B),C)$ and natural isomorphisms $\lambda_A:
  \mu(A,\eta(*)) \rto A$ and $\rho_A:\mu(\eta(*),A) \rto A$ (as well as a
  natural isomorphism $\gamma_{A,B}:\mu(A,B) \rto \mu(B,A)$, in the symmetric
  case) satisfying the usual relations of a symmetric monoidal category.

  A symmetric monoidal morphism of symmetric monoidal assemblers $F:\C \rto \D$ is
  a morphism of assemblers together with a natural transformation $\nu_{A,B}:
  \mu_\D(F(A), F(B)) \rto F(\mu_\C(A,B))$ and a morphism $\epsilon: \eta_\D(*)
  \rto F(\eta_\C(*))$ satisfying the relations of a symmetric monoidal functor.
\end{definition}

We begin by showing that the $K$-theory of a symmetric monoidal assembler is an
$E_\infty$-ring spectrum.

\begin{definition}
  For any sequence of objects
  $[A_1,\ldots,A_n]$ in $\C$, define
  \[M[A_1,\ldots,A_n] \defeq \mu(A_1,\mu(A_2,\ldots,\mu(A_{n-1},A_n))).\] Define
  $M[] = \eta(*)$.  Let $\hat \C$ be the assembler with
  \begin{description}
  \item[noninitial objects] finite sequences $[A_1,\ldots,A_n]$ of noninitial
    objects in $\C$, and
  \item[morphisms] given by
    \[\Hom_{\hat \C}([A_1,\ldots,A_n], [B_1,\ldots,B_m]) =
      \Hom_\C(M([A_1,\ldots,A_n]), M([B_1,\ldots,B_m])).\]
  \item[topology] covering families are those which are mapped to covering
    families by $M$.
  \end{description}
  Thus $M$ extends to a functor $M:\hat \C \rto \C$ which is essentially
  surjective, full, and faithful.  We can then define the topology by defining
  the covering families to be the preimages of covering families under $M$.
  Then $M$ is a continuous equivalence of families; the inverse equivalence is
  given by the functor mapping $A$ to the sequence $[A]$---except for objects
  isomorphic to $\eta(*)$, which are mapped to $[]$.

  We define a symmetric monoidal structure on $\hat \C$ on objects by defining
  \[\mu([A_1,\ldots,A_n],[B_1,\ldots,B_m]) = [A_1,\ldots,A_n,B_1,\ldots,B_m].\]
  On morphisms, we define $\mu(f:A \rto A',g:B \rto B')$ to be the morphism given by
  \[\mu(M(A),M(B)) \rto^{\cong} M(\mu(A,B)) \rto^{\mu(f,g)} M(\mu(A',B')) \rto^{\cong}
    \mu(M(A'),M(B')).\] The two marked isomorphisms are uniquely defined using
  the associator, so this is a well-defined morphism.  The associator and the
  two projections are identities.  The symmetry map
  \[\gamma_{A,B}: \mu([A_1,\ldots,A_n], [B_1,\ldots,B_m]) \rto
    \mu([B_1,\ldots,B_m], [A_1,\ldots,A_n]\]
  is induced by the unique shuffle map
  \[\mu(M(A),M(B)) \rto^{\cong} \mu(M(B), M(A)).\]
\end{definition}

\begin{lemma} \label{lem:tomonoid}
  For a (symmetric) monoidal assembler $\C$, the functor $M$ induces a homotopy
  equivalence $K(\hat \C) \rto K(\C)$.  In addition, $\hat \C$ is a monoid
  object in $\Assemb$.
\end{lemma}

\begin{proof}
  By definition, $\C$ is the full subassembler of $\hat \C$ containing all
  sequences of length at most $1$.  As this is an equivalence of categories it
  induces an equivalence $\W(\C) \rto \W(\hat C)$, and thus induces an equivalence
  on $K$-theories, as desired.
  
  We now check that $\hat \C$ is a monoid object.  First we show that
  $\mu$ and $\eta$ are well-defined morphisms of assemblers.  For $\eta$ this is
  clear: since $\S^\Assemb$ has no nontrivial covering families and no disjoint
  objects, this simply states that the map $\S^\Assemb \rto \hat \C$ taking $*$
  to $[]$ is a functor.  

  Now consider $\mu$.  It preserves the initial object by definition.  To check
  that it preserves covering families it suffices to check that it preserves all
  families in the coverage.  Thus consider a covering family
  $\{(A_i,B_j) \rto (A,B)\}_{(i,j)\in I\times J}$ in the coverage generating the
  topology of $\hat \C \sma \hat \C$.  By definition, this implies that
  $\{M(A_i) \rto M(A)\}_{i\in I}$ and $\{M(B_j) \rto M(B)\}_{j\in J}$ are
  covering families in $\C$, and thus $\{\mu(M(A_i),M(B_j)) \rto
  \mu(M(A),M(B))\}_{(i,j)\in I\times J}$ is a covering family in $\hat \C$.  By
  definition, the associator induces an isomorphism
  \[\mu(M(X),M(Y)) \rto M(\mu(X,Y))\]
  for all $X$ and $Y$.  Thus the covering family
  $\{\mu(M(A_i),M(B_j)) \rto \mu(M(A),M(B))\}_{(i,j)\in I\times J}$ implies that
  $\{M(\mu(A_iB_j)) \rto M(\mu(A,B))\}_{(i,j)\in I\times J}$ is a covering
  family as well.  Thus $\mu$ preserves covering families.  To check that it
  preserves disjointness note that 

  Since $\hat \C$ is strictly associative and strictly unital the relevant
  commutative diagrams in $\Assemb$ commute on-the-nose, and $\hat \C$ is a
  monoid object.
\end{proof}

Together with \cite[Corollary 3.9]{elmendorfmandell}, the above results prove
the following theorem.

\begin{theorem} \label{thm:Einfty}
  Let $\C$ be a symmetric monoidal assembler.  Then $K(\C)$ is equivalent to a
  strictly commutative ring symmetric spectrum.
\end{theorem}

\begin{example}
  Let $\Var_S$ be the assembler of varieties over $S$, and define the monoidal
  structure by $\mu(X,Y) \defeq X\times_S Y$.  Then $K(\Var_S)$ is an
  $E_\infty$-ring spectrum.
\end{example}

We now turn our attention to symmetric monoidal morphisms.

\begin{theorem} \label{thm:ringhom}
  A symmetric monoidal morphism of symmetric monoidal assemblers induces a ring
  homomorphism on $K$-groups.
\end{theorem}

\begin{proof}
  The main issue is the question of how strict the symmetric monoidal morphism
  is.  Let $F: \C \rto \D$ be a symmetric monoidal morphism of symmetric
  monoidal assemblers.  Using strictification, there is a diagram
  \begin{diagram}
    { \hat \C & \hat \D \\ \C & \D. \\};
    \arrowsquare{\hat F}{}{}{F}
    \diagArrow{implies, -implies}{2-1}{1-2}
  \end{diagram}
  commuting up to natural isomorphism; thus the map $K_*(\C) \rto K_*(\D)$ is
  the same as the composition
  \[K_*(\C) \rto K_*(\hat \C) \rto K_*(\hat \D) \rto K_*(\D).\]
  The functor $\W(\hat F)$ is strict on
  the additive and multiplicative structure, since $\hat F$ is strict on the
  multiplicative structure by definition, and the rest of the structure arises
  from the structure on $\FinSet$, which is the same in both domain and
  codomain.  Thus by \cite[Theorem 9.3.7]{elmendorfmandell} it is equivalent to
  a map of strictly commutative ring spectra, and thus induces a ring
  homomorphism on homotopy groups.  Since the diagram commutes up to natural
  isomorphism, it commutes up to homotopy after applying $K$-theory\note{CHECK
    THIS!!!}; in particular, the homomorphisms induced on $\pi_*$ must be the
  same.  Thus, since the vertical maps are ring isomorphisms and the top map is
  a ring homomorphism, the induced map $K_*(\C) \rto K_*(\D)$ must also be a
  ring homomorphism.
\end{proof}

With this theorem we now have the two desired examples:
\begin{example}
  Let $\Var_k$ be the assembler of $k$-varieties for $k$ a finite field, and let
  $G = \Gal(\bar k/k)$.  Then the map $\zeta: \Var_k \rto \AFSet_G$ given by $X
  \rgoesto X(\bar k)$ is a strong monoidal map of assemblers, and thus induces
  an $E_\infty$-map on the $K$-theory.
\end{example}

\begin{example}
  Let $f: T \rto S$ be a map of finite type.  Then base change induces an
  $E_\infty$-map of $K$-theories $K(\Var_S) \rto K(\Var_T)$.  In particular, for
  $K$ a number field with $k$ a residue field of its ring of integers, the map
  $K(\Var_{\Spec \mathcal{O}_K}) \rto K(\Var_k)$ is an $E_\infty$-map of spectra.
\end{example}

\section{Interaction with $K_1$} \label{sec:K1}

A model for the $K_1$ of an assembler was given in \cite{Z-ass-pi1}, which
contains the following theorem:
\begin{theorem}[{\cite[Theorem B]{Z-ass-pi1}}] \label{thm:Bpi1}
  For any assembler $\C$, every element of 
  $K_1(\C)$ can be represented by a pair of morphisms
  \[\morpair{A}{B}{f}{g}\]
  in $\W(\C)$.  These satisfy the relations
  \[
    \big[\morpair{A}{B}{f}{f}\big] = 0,   \qquad 
    \big[\morpair{B}{C}{g_1}{g_2}\big] + \big[\morpair{A}{B}{f_1}{f_2}\big] =
    \big[\longmorpair{3em}{A}{C}{g_1f_1}{g_2f_2}\big]  \]
  and 
  \[{\noHorizArrowDetection
      \big[\morpair{A}{B}{f_1}{f_2}\big] + \big[\morpair{C}{D}{g_1}{g_2}\big] =
      \big[\begin{tikzpicture}[baseline] \node (A) at (0,0) {$A\amalg C$}; \node
        (B) at (5em,0) {$B\amalg D$}; \diagDrawArrow{->, bend left}{A}{B}^{f_1\amalg g_1} \diagDrawArrow{->, bend  right}{A}{B}_{f_2\amalg g_2}%
      \end{tikzpicture}\big]
    }
  \]
\end{theorem}

The particular case of interest in this paper is when $A = B$ is an object of
$\C$ (which can be considered objects of $\W(\C)$ by indexing over a singleton)
and $g$ is the identity morphism.  Then $f$ must be an automorphism of $A$.
We write the generator
\[[A,f] \defeq \big[\morpair AAf{1_A}\big].\]
When restricted to such generators, the above theorem implies the following:
\begin{corollary}\label{cor:K1sp}
  Let $\C$ be an assembler which contains disjoint unions, in the sense that for
  any two objects $A$ and $B$, the pushout of the diagram
  \[A \lto \initial \rto B\]
  exists and has a covering family by the induced inclusions from $A$ and $B$;
  we denote this by $\amalg$.
  
  Every element $[A, f]$ with $A\in \ob\C$ and $f\in \Aut(A)$ represents an
  element in $K_1(\C)$.  These elements sastisfy the relations
  \[[A, 1_A] = 0, \qquad [A, g] + [A, f] = [A, gf] \qquad [A, f] + [B, g] =
    [A\amalg B, f\amalg g].\]
\end{corollary}

\begin{proof}
  The first two relations follow directly from Theorem~\ref{thm:Bpi1} so we
  focus on proving the last case.  The $\amalg$ in the statement of the theorem
  is the disjoint union in $\W(\C)$, which simply takes disjoint unions on
  indexing sets, so it implies that
  \[[A,f] + [B,g] = \big[\morpair{\{A,B\}}{\{A,B\}}{f\amalg g}{1}\big].\] It
  remains to check that the representative on the right is
  $[A\amalg B, f \amalg g]$.
  
  Let $f\amalg g: A \amalg B \rto A \amalg B$ be the morphism on the disjoint
  union in $\C$.  Let $(f,g): \{A,B\} \rto \{A,B\}$ be the morphism induced by
  $f$ and $g$ inside $\W(\C)$.  Let $\delta: \{A,B\} \rto \{A\amalg B\}$ be the
  morphism given by the covering family $\{A \rto A\amalg B, B \rto A\amalg B\}$
  induced from the coproduct in $\C$.  Inside $\W(\C)$, the square
  \begin{diagram}
    { \{A,B\} & \{A,B\} \\ \{A\amalg B\} & \{A \amalg B\} \\};
    \arrowsquare{(f,g)}{\delta}{\delta}{f\amalg g}
  \end{diagram}
  commutes. Therefore in $K_1(\C)$,
  \begin{align*}
    \big[\morpair{\{A,B\}}{\{A,B\}}{f\amalg g}{1}\big] &= 
    \big[\morpair{\{A,B\}}{\{A,B\}}{f\amalg g}{1}\big] +
                                                         \big[\morpair{\{A,B\}}{\{A\amalg B\}}{\delta}{\delta}\big]\\
    &= \big[\morpair{\{A,B\}}{\{A\amalg B\}}{\delta}{\delta}\big] +
    \big[\morpair{A\amalg B}{A \amalg B}{f\amalg g}{1_{A\amalg B}}\big] =
      [A\amalg B, f\amalg g],
  \end{align*}
  as desired.
\end{proof}

We now turn to examining how these interact with the product structure.  From
the previous section we know that for any pair of closed assemblers $\C$ and
$\D$, there is a map
\[\phi:K(\C) \sma K(\D) \rto K(\C \sma \D).\]
Applying $\pi_1$ induces a map 
\[\big(K_0(\C) \otimes K_1(\D)\big) \oplus \big(K_1(\C) \otimes K_0(\D)\big)
  \rcofib \pi_1(K(\C) \sma K(\D)) \rto^{\pi_1\phi} K_1(\C\sma \D).\] Since
$K_0(\C)$ is generated by symbols $[C]$, where $C\in \ob\SC(\C)$, and $K_1(\D)$
is generated by symbols $\big[\morpair{A}{B}{f}{g}\big]$, it makes sense to ask
whether there is a good representative for
\[\phi\left([C] \otimes \big[\morpair{A}{B}{f}{g}\big]\right).\]

\begin{theorem} \label{thm:K1mult-real} Let $\C$ and $\D$ be closed assemblers, and
  let $\phi:K^W(\C) \sma K^W(\D) \rto K^W(\C\sma \D)$ be given by the monoidal
  structure of $K^W$.  For any $[C]\in K_0(\C)$ and
  $\big[\morpair{A}{B}{f}{g}\big]$ in $K_1(\D)$,
  \[\phi_*\left( [C] \otimes \big[\morpair ABfg\big] \right) = \left[
      \longmorpair{1.2}{(C,A)}{(C,B)}{(1_C,f)}{(1_C,g)} \right] \in K_1(\C\sma
    \D).\]
  Analogously, for any $\big[\morpair CDfg\big]$ in $K_1(\C)$ and $[A]\in
  K_0(\D)$, 
  \[\phi_*\left( \big[\morpair CDfg\big]\otimes [A] \right) = \left[
      \longmorpair{1.2}{(C,A)}{(D,A)}{(f,1_A)}{(g,1_A)} \right] \in K_1(\C\sma \D).\]

\end{theorem}

\begin{proof}
  In \cite[Theorem 2.5]{murotonks07} Muro and Tonks give a presentation for the
  $1$-type of a Waldhausen category in a way which is compatible with its
  multiplicative structure.  More concretely, they show that for any Waldhausen
  category $\E$ there exists a stable quadratic module (see \cite[Definition
  1.4]{murotonks07}) generated by the objects and weak equivalences in $\E$,
  whose homology encodes $K_0(\E)$ and $K_1(\E)$.  This is compatible with the
  multiplicative structures on Waldhausen categories, in the sense that for a
  biexact functor $F: \C\times \D \rto \E$ of Waldhausen categories, the
  generators map in a predictable manner, with (for example)
  $F_*([A], [C \rwe D]) = [F(A,C) \rwe F(A,D)]$.  Muro and Tonks then give
  generators and relations for $K_1$ of a Waldhausen category based on the
  structure of the stable quadratic module.

  In \cite[Definition 3.2, Proposition 3.4]{Z-ass-pi1} a special case of the
  structure is worked out, in the case when the Waldhausen structures arise from
  assemblers; generators and relations for $K_0$ and $K_1$ of an assembler
  follow from this in \cite[Theorem 3.8]{Z-ass-pi1}.  The map $\phi$ in the
  statement of the theorem arises from the biexact functor
  $\SC(\C) \times \SC(\D) \rto \SC(\C\sma \D)$ induced by $\nu_{\C,\D}$.
  Applying the multiplicative structure given by Muro and Tonks to the
  presentations of $K_0$ and $K_1$ gives the desired result. 
\end{proof}

\begin{remark}
  The map on $\pi_1$ also induces a map
  \[\operatorname{Tor}_1(K_0(\C), K_0(\D)) \rto K_1(\C\sma \D).\]
  If this term is nontrivial it may be possible to find an interesting
  interpretation for its image in $K_1(\C\sma\D)$. 
\end{remark}

\begin{corollary}
  Let $\C$ and $\D$ be closed assemblers.
  For any objects $A\in \C$, $B\in \D$, $\alpha\in \Aut(A)$ and $\beta\in \Aut(B)$,
  \[\phi_*([A]\otimes [B,\beta]) = [(A,B), (1, \beta)]\in K_1(\C\sma\D)\]
  and
  \[\phi_*([A,\alpha]\otimes [B]) = [(A,B), (\alpha, 1)] \in K_1(\C\sma \D).\]
\end{corollary}

\bibliographystyle{IZ}
\bibliography{IZ-all}

\begin{thebibliography}{Zak17A}

\bibitem[Cam15]{campbell14}
Jonathan~A. Campbell.
\newblock The k-theory spectrum of varieties.
\newblock {\em Transactions of the American Mathematical Society}, 05 2015.

\bibitem[Car95]{carlsson95}
Gunnar Carlsson.
\newblock On the algebraic {$K$}-theory of infinite product categories.
\newblock {\em K-Theory}, 9:305--322, 07 1995.

\bibitem[CWZ19]{CWZ-zeta}
Jonathan Campbell, Jesse Wolfson, and Inna Zakharevich.
\newblock Derived {$\ell$}-adic zeta functions.
\newblock {\em Advances in Mathematics}, 354:106760, 2019.

\bibitem[DS88]{dresssiebeneicher88}
Andreas W.~M. Dress and Christian Siebeneicher.
\newblock The {B}urnside ring of profinite groups and the {W}itt vector
  construction.
\newblock {\em Adv. in Math.}, 70(1):87--132, 1988.

\bibitem[DS89]{dresssiebeneicher89}
Andreas~W.M Dress and Christian Siebeneicher.
\newblock The burnside ring of the infinite cyclic group and its relations to
  the necklace algebra, {$\lambda$}-rings, and the universal ring of witt
  vectors.
\newblock {\em Advances in Mathematics}, 78(1):1--41, 1989.

\bibitem[EM06]{elmendorfmandell}
A.~D. Elmendorf and M.~A. Mandell.
\newblock Rings, modules, and algebras in infinite loop space theory.
\newblock {\em Adv. Math.}, 205(1):163--228, 2006.

\bibitem[Gra78]{grayson78}
Daniel~R. Grayson.
\newblock Grothendieck rings and {W}itt vectors.
\newblock {\em Comm. Algebra}, 6(3):249--255, 1978.

\bibitem[JY]{johnsonyau}
Niles Johnson and Donald Yau.
\newblock Bimonoidal categories, en-monoidal categories, and algebraic
  k-theory.
\newblock https://nilesjohnson.net/En-monoidal.html.

\bibitem[KW20]{kasprowskiwinges20}
Daniel Kasprowski and Christoph Winges.
\newblock Shortening binary complexes and commutativity of {$K$}-theory with
  infinite products.
\newblock {\em Trans. Amer. Math. Soc. Ser. B}, 7:1--23, 2020.

\bibitem[MT07]{murotonks07}
Fernando Muro and Andrew Tonks.
\newblock The 1-type of a {W}aldhausen {$K$}-theory spectrum.
\newblock {\em Adv. Math.}, 216(1):178--211, 2007.

\bibitem[Wei13]{kbook}
Charles~A. Weibel.
\newblock {\em The {$K$}-book}, volume 145 of {\em Graduate Studies in
  Mathematics}.
\newblock American Mathematical Society, Providence, RI, 2013.
\newblock An introduction to algebraic $K$-theory.

\bibitem[Yau19]{yau19}
Donald Yau.
\newblock Infinity operads and monoidal categories with group equivariance,
  2019.

\bibitem[Zak12]{zakharevich10}
Inna Zakharevich.
\newblock Scissors congruence as {$K$}-theory.
\newblock {\em Homology, Homotopy and Applications}, 14:181--202, 2012.

\bibitem[Zak17A]{Z-Kth-ass}
Inna Zakharevich.
\newblock The {$K$}-theory of assemblers.
\newblock {\em Advances in Mathematics}, 304:1176--1218, 2017.

\bibitem[Zak17B]{Z-ass-pi1}
Inna Zakharevich.
\newblock On {$K_1$} of an assembler.
\newblock {\em J. Pure Appl. Algebra}, 221(7):1867--1898, 2017.

\end{thebibliography}

\end{document}